\documentclass[11pt,a4paper]{article}
\usepackage[T1]{fontenc}
\usepackage[utf8]{inputenc}
\usepackage[english]{babel}
\usepackage{amsmath}
\usepackage{amsfonts}
\usepackage{amssymb}
\usepackage{amsthm}
\usepackage{makeidx}
\usepackage{mathrsfs}
\usepackage{multirow}
\usepackage{color}
\usepackage{colortbl}
\usepackage{hyperref}
\usepackage{array}
\usepackage{bigstrut}
\usepackage{setspace}

\usepackage[pdftex]{graphicx}
\usepackage{bm}
\numberwithin{equation}{section}

\theoremstyle{definition}
\newtheorem{definition}{Definition}[section]

\theoremstyle{plain}
\newtheorem{theorem}[definition]{Theorem}

\newtheorem{proposition}{Proposition}

\theoremstyle{remark}
\newtheorem{remark}{Remark}

\date{}

\begin{document}  

\author{
\large{\bf{Donatella Donatelli}} \\[1ex] 
\normalsize Department of Information Engineering, Computer Science and Mathematics \\ 
\normalsize University of L'Aquila \\
\normalsize 67100 L’Aquila, Italy. \\
\normalsize \href{mailto:donatella.donatelli@univaq.it}{donatella.donatelli@univaq.it}
\and 
\large{\bf{Pierangelo Marcati}} \\[1ex] 
\normalsize GSSI - Gran Sasso Science Institute \\ 
\normalsize Viale F.\,Crispi, 7 \\
\normalsize 67100 L’Aquila, Italy.\\
\normalsize \href{mailto:pierangelo.marcati@univaq.it}{pierangelo.marcati@univaq.it}, \href{mailto:pierangelo.marcati@gssi.infn.it}{pierangelo.marcati@gssi.infn.it}
\and
\large{\bf{Licia Romagnoli}} \\[1ex] 
\normalsize Department of Information Engineering, Computer Science and Mathematics \\ 
\normalsize University of L'Aquila \\
\normalsize 67100 L’Aquila, Italy.\\
\normalsize \href{mailto:licia.romagnoli@graduate.univaq.it}{licia.romagnoli@graduate.univaq.it}
}

\title{\normalfont{Analysis of solutions for a cerebrospinal fluid model}} 
\maketitle

\begin{abstract}
The aim of this manuscript is to analyze an intracranic fluid model from a mathematical point of view. By means of an iterative process we are able to prove the existence and uniqueness of a local solution and  the existence and uniqueness of a global solution under some restriction conditions on the initial data. Moreover the last part of the paper is devoted  to present numerical simulations for the analyzed cerebrospinal model. In particular, in order to assess the reliability of the  stated  theoretical results, we carry out the numerical simulations in two different cases: first, we fix initial data which  \,satisfy the conditions for the global existence of solutions, then, we choose initial data that violate them.
\end{abstract}

\section{Introduction}

The main purpose of the present paper is to analyze an intracranic fluid model from a mathematical point of view. One of the main difficulties is related to the complexity of the intracranial dynamics which is origin of many different phenomena: the flow of the cerebrospinal fluid throughout the CSF compartments, the mechanical interaction between the fluid and the brain, the physiology of the brain, coupling with the circulatory-system and production and reabsorption laws.\\
There is a lack in literature of rigorous mathematical results regarding the analysis of the solutions to the systems of equations which model the different mechanisms in the intracranic pattern. Understanding the mathematical properties of the model is often intimately related to establish its validity in the analysis of the physiological properties.
Our aim here is then to provide a comprehensive study of the equations that rule the cerebrospinal fluid (CSF) dynamics.\\
The physiological process we are interested in originates in the choroid plexus (see Fig. \ref{fig1}), affected by the cardiac cycle, that produces CSF at a constant rate (approximately 500 ml per day): the systole induces an expansion of the choroid that acts like a driving force for the CSF motion. Once secreted the CSF flows through four ventricles linked by different foramina and finally it reaches the subarachnoidal space, that is connected to the fourth ventricle by the foramina of Luschke and Magendie (see Fig. \ref{fig2}). At this point CSF is absorbed across the arachnoid villi into the venous circulation and a significant amount probably also drains into lymphatic vessels around the cranial cavity and spinal canal. The arachnoid villi act as one-way valves between the subarachnoid space and the dural sinuses. 
\begin{figure}[h]\label{fig1}
\centering
\includegraphics[scale=0.45]{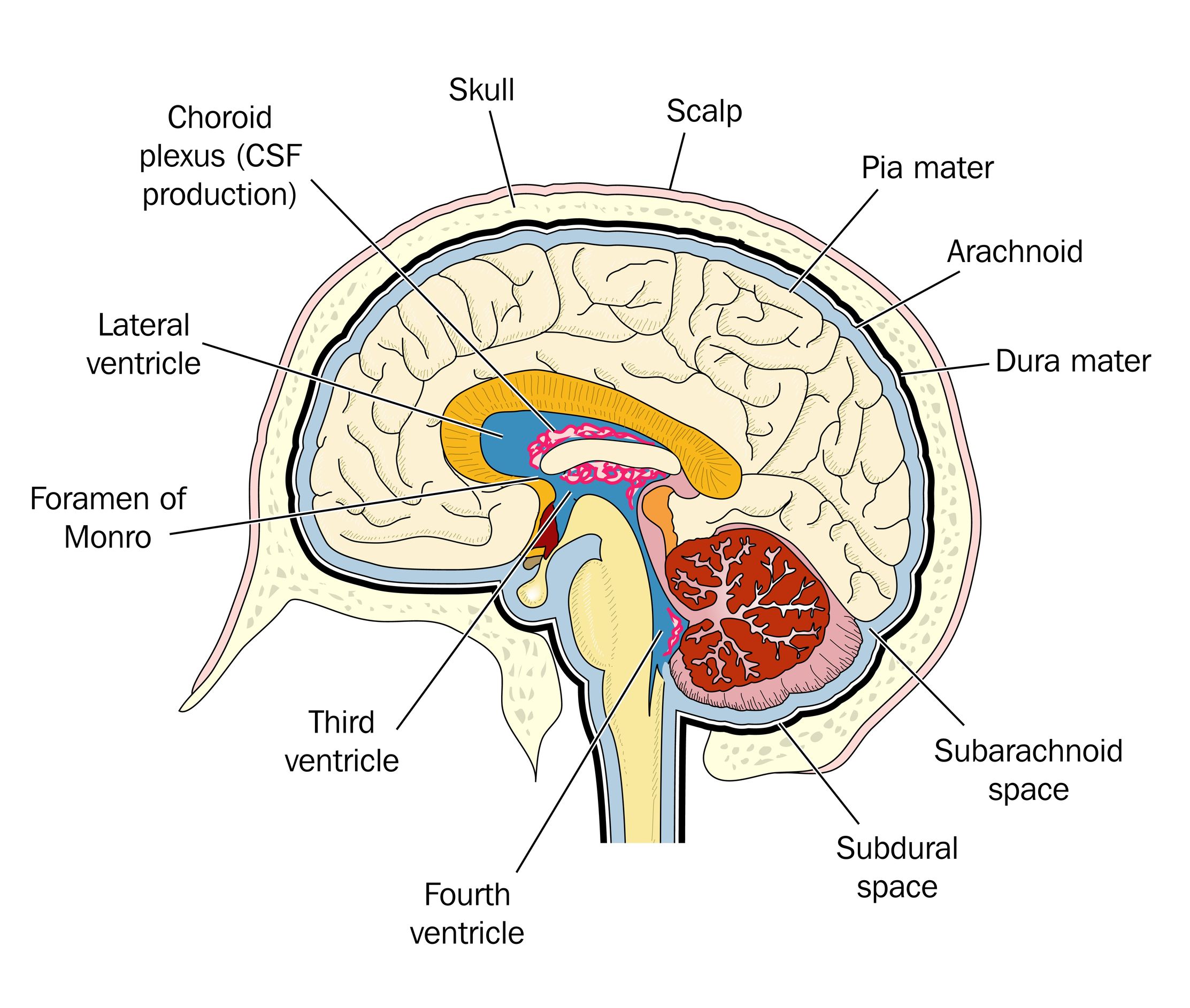}
\caption{\small{Cross section of the human brain.}}
\end{figure}

In order to make the cerebral dynamics more comprehensible, detailed mathematical models quantifying forces and their interaction have become fundamental to reveal what medical instruments are not able to do without affecting the data.\\
In fact since the mid seventies, the mathematical modeling of the intracranial dynamics has gained interest and many researchers have developed and analyzed models of different complexity for each of the three compartments (brain parenchyma, cerebral vasculature and cerebrospinal fluid), either standing alone or coupled. An extensive overview of the models for intracranial dynamics may be found in \cite{b2}, \cite{b1}. \\
More precisely, in order to link the mathematical models for the intracranial dynamics it is possible to identify two classes: models based on partial differential equations and models based on a lumped compartmental description of the intracranial constituents.

Linninger et al.\,(see \cite{b3}) developed a lumped model for describing the CSF pulsatility during the cardiac cycle. This first model included the main constituents present in the cranial vault, however only the CSF compartment was fully treated. In fact the model includes the lateral ventricles, the third ventricle, the fourth ventricle and the subarachnoid space. 
\begin{figure}[h]\label{fig2}
\centering
\includegraphics[scale=0.4]{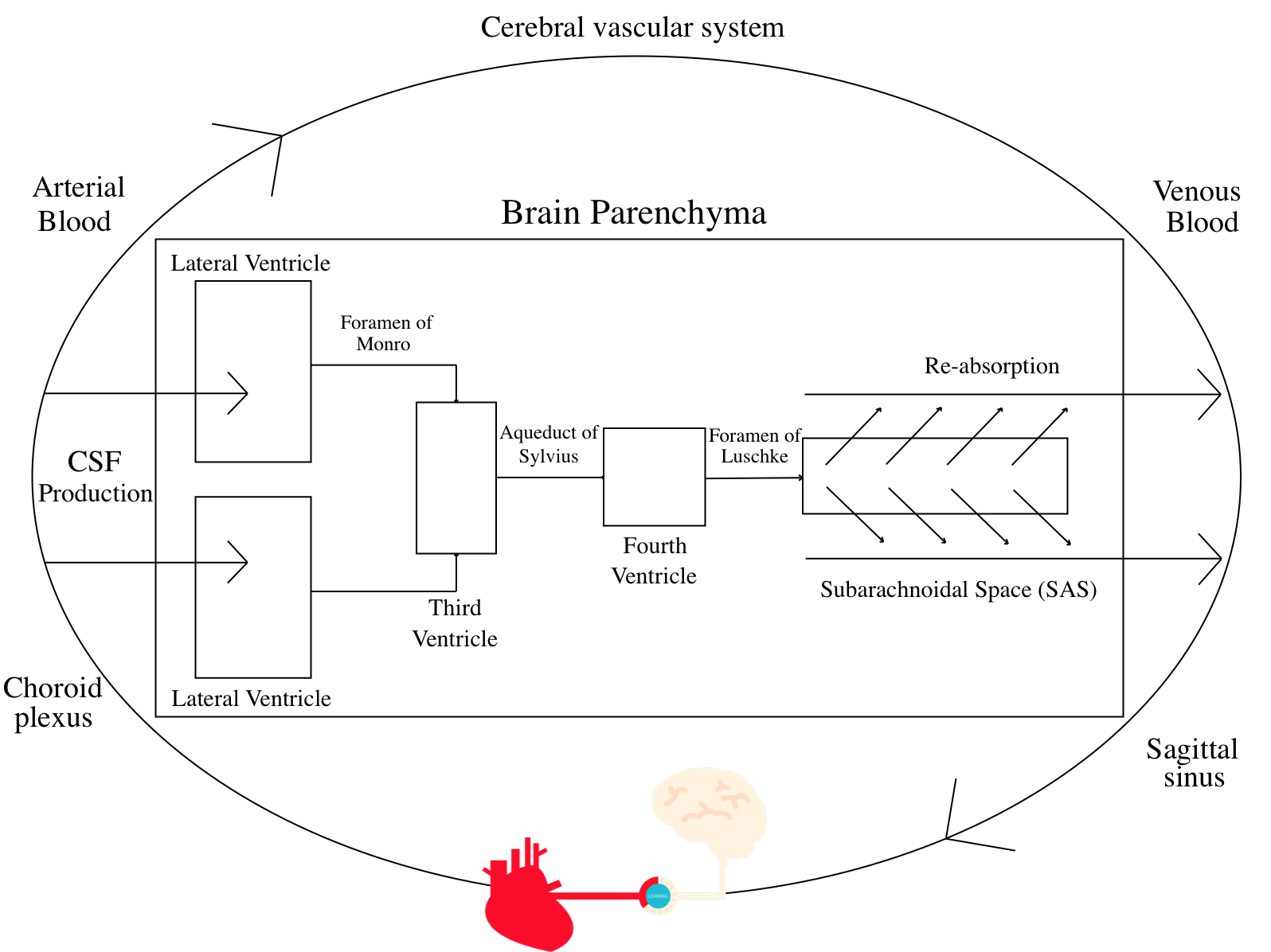}
\caption{\small{Scheme of the intracranic interacting systems: the CSF environments, the brain parenchyma and the vascular system.}}
\end{figure}
The cerebral blood compartment is taken into account only as boundary condition for the CSF system. Indeed, the CSF pulsation is driven by a boundary condition modeling the effects of the choroid plexus which expands and contracts in a prescribed way (modeling the effect of the arterial pressure wave) and the CSF reabsorption is modeled by an equation where the venous pressure is a prescribed constant. The brain parenchyma is also treated in a simplified way. In fact equations to link the CSF compartments pressure and volume are obtained by supposing that each compartment is a cylinder with a thin shell membrane that deforms radially under the pressure difference between the CSF and the parenchyma. 

In this paper we adopt a different approach to the study of the cerebrospinal fluid dynamics, in fact we start from the Linninger model by retracing and simplifying the more important steps of the modeling and we provide a comprehensive analysis of the system of equations that represents the pillar of the aforementioned dynamics, in order to fill up the gap between the mathematical theory and modelization.\\

First of all we clarify the system of equations that describes the CSF hydrodynamics and we obtain the following system
\begin{eqnarray}\label{ar1}
\begin{cases}
\partial_t\eta(t,z)+\partial_t a(t)+u(t,z)-\tilde{Q}_p=0, \\
\alpha\partial_{tt}\eta(t,z)+\tilde{k}\partial_t\eta(t,z)+\kappa\eta(t,z)-AP(t,z)+A\tilde{P}=0,  \\
\rho\partial_t u(t,z)+\rho u(t,z)\partial_z u(t,z)+\partial_z P(t,z)+\beta u(t,z)=0,  
\end{cases}
\end{eqnarray}
where $u$ is the CSF velocity flux, $\eta$ is the tissue displacement and $P$ is the pressure. In this first approach we neglect the variation of the cross section $A$ due to the choroid plexus expansion that drives the pulsatile CSF circulation, then we consider it constant in order to study the effects of this statement on the real cerebral physiology. Moreover we choose properly the initial and boundary conditions with the related compatibility conditions that are all treated in detail in Section \ref{paragr1}. \\
Then we proceed by proving the local existence and uniqueness of a solution to the system \eqref{ar1}. This result is achieved by setting up a proper iterative scheme based on the approximating system strictly related to \eqref{ar1}. 
As a first step we prove by induction on the number $n$ of iterations the existence and uniqueness of a solution, $\mathcal{X}^n(t,z)=(u^n, \eta^n, P^n)(t,z)$, to the approximating system, then by using higher energy estimates ($\mathcal{H}^s, s>\frac{9}{2}$) we prove the convergence of the approximating sequence, $\mathcal{X}^n(t,z)$, to a local classical solution of \eqref{ar1}.

This first result allows us to investigate the global existence of a solution for the system \eqref{ar1}. Since $\eta(t,z)$ and $P(t,z)$ inherit the lifespan of $u(t,z)$, we focus on the third equation of \eqref{ar1} that we study first in the homogeneous form in order to obtain a detailed framework. We apply the characteristic method and we get in both cases a Riccati equation that we are able to solve with the standard ODE methods for the homogeneous case and by means of a particular solution that we construct properly for the nonhomogeneous case.\\
Finally, by using a sharp continuation principle (see \cite{majda}) we show that there exists a global solution to \eqref{ar1} if and only if it is satisfied a precise condition on the initial datum of $u$.

The global existence of a solution paves the way for another important result due mainly to the coupling of the CSF flux and the cardiac cycle. It is well known that systole and diastole produce a periodic cardiac movement that holds a big role in the intracranic dynamics and this suggests the existence of particular solutions for the system \eqref{ar1} if we consider it independent of the spatial variable $z$. Indeed we are able to prove the existence and uniqueness of periodic solutions that are strongly influenced by the forcing term $a(t)$ and whose stability is ensured by the global existence time.\\

The present paper is organized as follows. In section 2 we begin a detailed discussion with a description of the Linninger model and, once we have reformulated it in a more simple way, we fix the initial and boundary conditions. In section 3 we state our main results on the existence and uniquess of a local solution and on the existence and uniqueness of a global solution to the system of equations \eqref{ar1}. In section 4, we set up an iterative process by means of an approximating system associated to the \eqref{ar1}. In section 5, in order to discuss the convergenge of the iterative scheme we derive higher order energy estimates and we prove the convergence for each one of the sequences that appear in the approximating system. Then we conclude the proof of the local existence and uniqueness of a solution to \eqref{ar1} in section 6.
Section 7 is devoted to the problem of the global existence of solutions. Since the third equation of the system \eqref{ar1} has a ``Burgers'' like behaviour we focus first on that equation and then we deduce the global existence for the remaining equations. Hence we analyze in detail the global existence time in order to adopt a continuation principle that is needed for the proof of the second theorem stated in section 3. In section 8 we analyze the periodic system associated to \eqref{ar1} and we study the stability of its solutions. Finally, the last section of this paper is devoted to present numerical simulations for the analyzed cerebrospinal model. In particular, in order to assess the reliability of the results stated in section 3, we carry out the numerical simulations in two different cases: in the numerical method we fix first initial data which satisfy the conditions for the global existence of a solution to the system \eqref{ar1}, then we choose initial data that violate them. In the first case no blow up of the solutions is displayed and we can observe a parabolic behavior for the velocity profile and, as a consequence, a smooth evolution for the pressure and the tissue displacement. While, in the second case, simulations achieve blow up after a single iteration. These results allow us to validate the statements obtained in the present paper. 
\vspace{8pt}\\
\textbf{Notations.}\ \ We need to define the functional spaces that are necessary to study our problem.\\
Let $\Omega \subset\mathbb{R}$ be an open bounded set. For every $s\in\mathbb{N}$ we define:
\begin{displaymath}
C^0(\bar{\Omega})=\{f:\bar{\Omega}\rightarrow\mathbb{R}\, \lvert \, f\, \mbox{is continuous}\},
\end{displaymath}
\begin{displaymath}
C^s(\bar{\Omega})=\{f:\bar{\Omega}\rightarrow\mathbb{R}\,\lvert \ \ f \in C^0(\bar{\Omega})\ \ \mbox{and}\ \ \partial^{\alpha}f\in C^0(\bar{\Omega})\ \ \mbox{for any}\ \ \alpha\ \ \mbox{with}\ \ \alpha\leq s\},
\end{displaymath}
\begin{displaymath}
\mathcal{H}^s(\Omega)=\{v\in\mathcal{D}'(\Omega)\,\lvert \ \ \partial^{\alpha}v\in L^2(\Omega) \ \ \mbox{for all} \ \ \alpha\leq s\},
\end{displaymath}
where $\mathcal{D}'(\Omega)$ is the space of distributions in $\Omega$.\\
We will denote by $\left\|\cdot\right\|_{L^{\infty}_t\mathcal{H}^s_z}$ the norm $\left\|\cdot\right\|_{L^{\infty}((0,T))\mathcal{H}^{s}([0,L])}$ and by $\left\|\cdot\right\|_2$ the norm $\left\|\cdot\right\|_{L^2([0,L])}$\\
Moreover, we recall the Sobolev interpolation theorem by which, if $0\leq r'\leq r$, there exists a costant $C_r$ such that 
\begin{equation}\label{ar2}
\left\| w\right\|_{\mathcal{H}^{r'}([0,L])}\leq C_r\left\| w\right\|_{L^{2}([0,L])}^{1-\frac{r'}{r}}\left\| w\right\|_{\mathcal{H}^{r}([0,L])}^{\frac{r'}{r}} \ \ \mbox{for any}\ \ w\in\mathcal{H}^r([0,L]).
\end{equation}

\section{Mathematical model and governing equations}\label{paragr1}

In this paper we analyze the model introduced in \cite{b3} which simplifies the intracranic dynamics by representing the fluid-structure interactions that originate in the craniospinal pattern and by quantifying the cerebrospinal fluid (CSF) motion such as the variations of the pressure due to the cerebral tissue compression. In order to do that, the ventricles are discretized into cylindrical finite volume with perfect axial dispersion and radial \,expandability, the foramina that link the different compartments are treated as elastic tubes and the CSF flow is considered basically laminar.\\

The continuity equation of the cerebrospinal fluid flow in the ventricles and the axial momentum equation are given by
\begin{eqnarray}\label{ar3}
\begin{cases}
\displaystyle{\frac{\partial\{A[h+a(t)+\eta(t,z)]\}}{\partial t}}= Q_{p}-Q_f(t,z),\\
\rho\left[\displaystyle{\frac{\partial u(t,z)}{\partial t}}+u(t,z)\displaystyle{\frac{\partial u(t,z)}{\partial z}}\right]+\displaystyle{\frac{\partial P(t,z)}{\partial z}}= -\displaystyle{\frac{8\mu}{r^2}}u(t,z), 
\end{cases}
\end{eqnarray}
where
\begin{itemize}
\item $A$ and $h$ are the cross section and the height of the ventricular or subarachnoid section respectively;
\item $r$ is the radius of the foramina and aqueduct;
\item $\mu$ is the fluid viscosity;
\item $a(t)=\bar{\alpha}\left(1.3+\sin\left(\omega t-\frac{\pi}{2}\right)-\frac{1}{2}\cos\left(2\omega t-\frac{\pi}{2}\right)\right)$ is the forcing function;
\item $\eta(t,z)$ denotes the tissue displacement in a section;
\item $Q_{p}$ is the CSF production rate in the choroid plexus;
\item $Q_f(t,z)=Au(t,z)$ is the CSF flow rate leaving ventricle, i.e., flow in foramina and aqueduct;
\item $u(t,z)$ is the axial CSF flow velocity.
\end{itemize}
\begin{figure}
\centering
\includegraphics[scale=0.5]{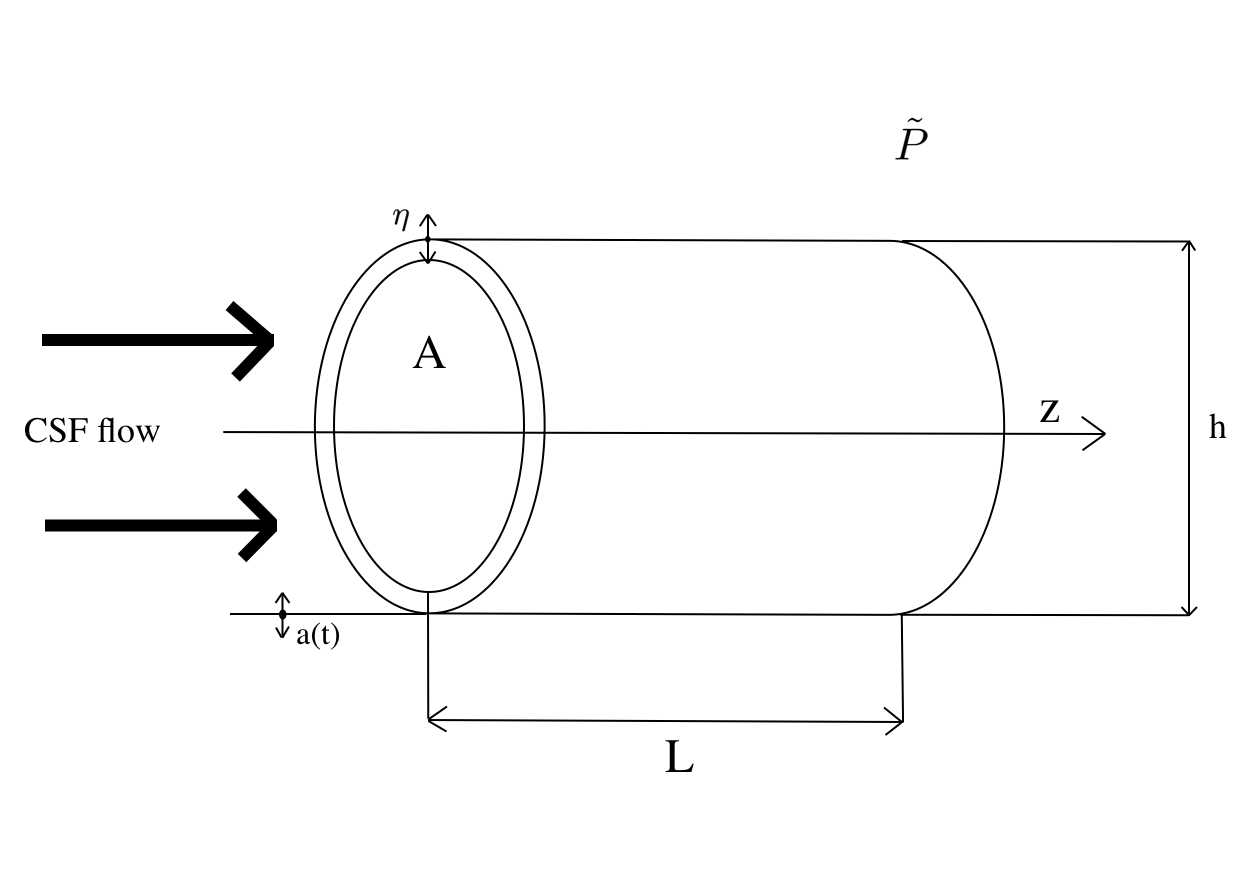}\label{fig3}
\caption{\small{Single compartment section of the CSF dynamics discretized model.}}
\end{figure}
The dynamics of parenchyma stresses, strains and displacements can be described with the laws of elastodynamics,
\begin{equation}\label{ar4}
(\rho_{\mathit{w}}A\delta)\ddot{\eta}(t,z)+k_d\dot{\eta}(t,z)+k_e\eta(t,z)-A[P(t,z)-\tilde{P}]=0, 
\end{equation}
where
\begin{itemize}
\item $k_d\dot{\eta}(t,z)$ is the dissipative force;
\item $k_e\eta(t,z)$ is the elastic force;
\item $P(t,z)-\tilde{P}$ is the difference between fluid and parenchyma pressure.
\end{itemize}
Since the deformation of the tissue directly affects the space available to the fluid, the two systems are fully coupled.\\
This model allows us to compute the pressure $P(t,z)$ in the cerebral ventricles, the velocity $u(t,z)$ of the fluid along the foramina and the tissue displacements $\eta(t,z)$ at each point.\\
The simplified equations \eqref{ar3}, \eqref{ar4} of motion for the hydrodynamics of the 1D intracranial dynamics can be written together in the following system:
\begin{eqnarray}\label{ar5}
\begin{cases}
\partial_t\eta(t,z)+\partial_t a(t)+u(t,z)-\tilde{Q}_p=0, \\
\alpha\partial_{tt}\eta(t,z)+\tilde{k}\partial_t\eta(t,z)+\kappa\eta(t,z)-AP(t,z)+A\tilde{P}=0,  \\
\rho\partial_t u(t,z)+\rho u(t,z)\partial_z u(t,z)+\partial_z P(t,z)+\beta u(t,z)=0,  
\end{cases}
\end{eqnarray}
where $t\in[0,T]$, $z\in[0,L]$, $a(t)\in C^{\infty}([0,T])$, $\tilde{Q}_p=\displaystyle{\frac{Q_p}{A}}$, $\rho, \beta, \tilde{k}, \kappa, \alpha, A, \tilde{P}\in\mathbb{R}$ are physical constants and known values.\\
First of all, for the unknown $u$ we impose the following initial condition  
\begin{equation}\label{ar6}
u(0,z)=u_0(z)=f(z)\,\in\,\mathcal{H}^{s}([0, L]),
\end{equation}
where $\displaystyle{s>\frac{9}{2}}$. \\
Since this paper is a first attempt to perform a rigorous mathematical \,analysis on the system \eqref{ar5} we start considering simplified boundary conditions neglecting some assumptions imposed by the realistic model. Hence for the single compartment of length $L$, we consider the following boundary conditions for $u(t,z)$:
\begin{equation}\label{ar7}
u(t,0)=u(t,L)=0,
\end{equation}
and the following compatibility conditions
\begin{align}\label{ar8}
u(0,0)&=u_0(0)=f(0), \notag \\
u(0,L)&=u_0(L)=f(L);
\end{align}
moreover we assume for the sake of simplicity
\begin{equation}\label{ar9}
f(0)=f(L)=0.
\end{equation}
Similarly for the unknown $\eta$ of the system \eqref{ar5}, we assume
\begin{equation}\label{ar10}
\eta(0,z)=\eta_0(z)=g(z)\in\mathcal{H}^{s}([0,L])
\end{equation}
and the following compatibility conditions
\begin{align}\label{ar11}
\eta(0,0)&=\eta_0(0)=g(0),  \notag \\
\eta(0,L)&=\eta_0(L)=g(L).
\end{align}
Now we want to compute the initial and boundary conditions for the pressure, $P$.
With the previous assumptions and by using the first two equations of our system we find that
\begin{align}\label{ar12}
\partial_t\eta(0,0)&=\omega\bar{\alpha}+\tilde{Q}_p, \notag \\
\partial_{tt}\eta(0,0)&=-\bar{\alpha}\omega^2, 
\end{align}
hence
\begin{equation}\label{ar13}
P(0,0)=s(0)=\tilde{P}-\frac{\bar{\alpha}\alpha\omega^2}{A}+\frac{\tilde{k}\bar{\alpha}\omega}{A}+\frac{\tilde{k}\tilde{Q}_p}{A}+\frac{\kappa}{A} g(0).
\end{equation}
In the same way we get
\begin{align}\label{ar14}
\partial_t\eta(0,z)&=\omega\bar{\alpha}+\tilde{Q}_p-f(z),\notag \\
\partial_t u(0,z)&=-f(z)\partial_z f(z)-\frac{1}{\rho}\partial_z P(0,z)-\frac{\beta}{\rho}f(z), \notag \\
\partial_{tt}\eta(0,z)&=-\bar{\alpha}\omega^2+f(z)\partial_z f(z)+\frac{1}{\rho}\partial_z P(0,z)+\frac{\beta}{\rho}f(z). 
\end{align}
By using the notation
\begin{equation}\label{ar15}
P(0,z)=s(z),
\end{equation}
and by replacing everything in the second equation of \eqref{ar5} we get
\begin{equation}\label{ar16}
s'(z)-\frac{A\rho}{\alpha}s(z)=h(z),
\end{equation}
where
\begin{align}\label{ar17}
h(z)=\rho\bar{\alpha}\omega^2-\rho f(z)\partial_z f(z)&-\beta f(z)-\frac{\rho}{\alpha}\bigg[\bar{\alpha}\tilde{k}\omega+\tilde{k}\tilde{Q}_p\notag \\
&-\tilde{k}f(z)+\kappa g(z)+A\tilde{P}\bigg].
\end{align}
The solution of the ODE \eqref{ar16} with the initial condition \eqref{ar13} is given by 
\begin{equation}\label{ar18}
P(0,z)=s(z)=\int_{0}^{z} e^{\frac{A\rho}{\alpha}(z-\zeta)}h(\zeta) \, d\zeta + s(0)e^{\frac{A\rho}{\alpha}z}.
\end{equation}
To establish the correct boundary conditions for the unknowns $\eta$ and $P$, we proceed as before with the assumptions \eqref{ar7}, \eqref{ar11} and the forcing function $a(t)$. Then we get
\begin{equation*}
\partial_t\eta(t,0)=-\partial_t a(t)+\tilde{Q}_p,
\end{equation*}
and from the previous one we obtain the first boundary condition for $\eta$   
\begin{equation}\label{ar19}
\eta(t,0)=g(0)-a(t)+0,3\bar{\alpha}+\tilde{Q}_pt.
\end{equation}
Moreover
\begin{align}\label{ar20}
\partial_{tt}\eta(t,0)&=-\partial_{tt}a(t)-\partial_t u(t,0)\notag \\
&=-\bar{\alpha}\left(-\omega^2\sin\left(\omega t-\frac{\pi}{2}\right)+2\omega^2\cos\left(2\omega t-\frac{\pi}{2}\right)\right),
\end{align}
and the first boundary condition for the unknown $P$ is given by
\begin{align}\label{ar21}
P(t,0)&= \frac{\bar{\alpha}\alpha\omega^2-\kappa\bar{\alpha}}{A}\sin\left(\omega t-\frac{\pi}{2}\right)+\frac{\kappa\bar{\alpha}-4\bar{\alpha}\alpha\omega^2}{2A}\cos\left(2\omega t-\frac{\pi}{2}\right)\notag\\
&-\frac{\bar{\alpha}\tilde{k}\omega}{A}\left(\cos\left(\omega t-\frac{\pi}{2}\right)+\sin\left(2\omega t-\frac{\pi}{2}\right)\right)+\frac{\kappa}{A}g(0)+\frac{\tilde{k}+\kappa t}{A}\tilde{Q}_p\notag\\
&-\frac{\kappa\bar{\alpha}}{A}+\tilde{P}.
\end{align}
Similarly we obtain the following boundary condition for $z=L$
\begin{align}\label{ar22}
\eta(t,L)&=g(L)-a(t)+0,3\bar{\alpha}+\tilde{Q}_pt, \notag \\
P(t,L)&=\frac{\bar{\alpha}\alpha\omega^2-\kappa\bar{\alpha}}{A}\sin\left(\omega t-\frac{\pi}{2}\right)+\frac{\kappa\bar{\alpha}-4\bar{\alpha}\alpha\omega^2}{2A}\cos\left(2\omega t-\frac{\pi}{2}\right)\notag\\
&-\frac{\bar{\alpha}\tilde{k}\omega}{A}\left(\cos\left(\omega t-\frac{\pi}{2}\right)+\sin\left(2\omega t-\frac{\pi}{2}\right)\right)+\frac{\kappa}{A}g(L)+\frac{\tilde{k}+\kappa t}{A}\tilde{Q}_p\notag\\
&-\frac{\kappa\bar{\alpha}}{A}+\tilde{P}.
\end{align}
Finally, we observe that further compatibility conditions are given by
\begin{align}\label{ar23}
\partial_t\eta(0,L)&=\omega\bar{\alpha}+\tilde{Q}_p,\notag \\
\partial_{tt}\eta(0,L)&=-\bar{\alpha}\omega^2,\notag \\
P(0,L)=s(L)&=\tilde{P}-\frac{\alpha\bar{\alpha}\omega^2+\bar{\alpha}\tilde{k}\omega+\kappa g(L)}{A}.
\end{align}

\section{Main results}
Now we are ready to state the main results of this paper. The local existence and uniqueness of a solution to the system \eqref{ar5} with the initial conditions \eqref{ar6}, \eqref{ar10}, \eqref{ar18} and boundary conditions \eqref{ar7}, \eqref{ar19}, \eqref{ar22}, is given by the following theorem.
\begin{theorem}\label{Teorema}
Let us consider the system
\begin{eqnarray}\label{ar24}
\begin{cases}
\partial_t\eta(t,z)+\partial_t a(t)+u(t,z)-\tilde{Q}_p=0, \\
\alpha\partial_{tt}\eta(t,z)+\tilde{k}\partial_t\eta(t,z)+\kappa\eta(t,z)-AP(t,z)+A\tilde{P}=0,  \\
\rho\partial_t u(t,z)+\rho u(t,z)\partial_z u(t,z)+\partial_z P(t,z)+\beta u(t,z)=0,  
\end{cases}
\end{eqnarray}
where $t\in[0,T_0]$, $z\in[0,L]$, $a(t)\in C^{\infty}([0,T_0])$, $\tilde{Q}_p=\displaystyle{\frac{Q_p}{A}}$, $\rho, \beta, \tilde{k}, \kappa, \alpha, A, \tilde{P}\in\mathbb{R}$ are constants, with initial conditions 
\begin{align}\label{ar25}
u(0,z)&=u_0(z)=f(z)\,\in\,\mathcal{H}^s([0, L]), \notag\\
\eta(0,z)&=\eta_0(z)=g(z)\in\mathcal{H}^s([0,L]),\notag \\
P(0,z)&=P_0(z)=s(z),
\end{align}
$\displaystyle{s>\frac{9}{2}}$, where 
\begin{align}\label{ar26}
s(z)&=\int_{0}^{z} e^{\frac{A\rho}{\alpha}(z-\zeta)}h(\zeta) \, d\zeta + s(0)e^{\frac{A\rho}{\alpha}z},\\
s(0)&=\tilde{P}-\frac{\bar{\alpha}\alpha\omega^2}{A}+\frac{\tilde{k}\bar{\alpha}\omega}{A}+\frac{\tilde{k}\tilde{Q}_p}{A}+\frac{\kappa}{A} g(0),
\end{align}
with
\begin{align}\label{ar26b}
h(z)=\rho\bar{\alpha}\omega^2-\rho f(z)\partial_z f(z)&-\beta f(z)-\frac{\rho}{\alpha}\bigg[\bar{\alpha}\tilde{k}\omega+\tilde{k}\tilde{Q}_p\notag \\
&-\tilde{k}f(z)+\kappa g(z)+A\tilde{P}\bigg].
\end{align}
Consider the following boundary conditions for the system \eqref{ar24}
\begin{align}\label{ar27}
u(t,0)&=u(t,L)=0,\notag\\
\eta(t,0)&=g(0)-a(t)+0,3\bar{\alpha}+\tilde{Q}_pt, \notag\\
\eta(t,L)&=g(L)-a(t)+0,3\bar{\alpha}+\tilde{Q}_pt,\notag \\
P(t,0)&= \frac{\bar{\alpha}\alpha\omega^2-\kappa\bar{\alpha}}{A}\sin\left(\omega t-\frac{\pi}{2}\right)+\frac{\kappa\bar{\alpha}-4\bar{\alpha}\alpha\omega^2}{2A}\cos\left(2\omega t-\frac{\pi}{2}\right)\notag\\
&-\frac{\bar{\alpha}\tilde{k}\omega}{A}\left(\cos\left(\omega t-\frac{\pi}{2}\right)+\sin\left(2\omega t-\frac{\pi}{2}\right)\right)+\frac{\kappa}{A}g(0)+\frac{\tilde{k}+\kappa t}{A}\tilde{Q}_p\notag\\
&-\frac{\kappa\bar{\alpha}}{A}+\tilde{P},\notag\\
P(t,L)&=\frac{\bar{\alpha}\alpha\omega^2-\kappa\bar{\alpha}}{A}\sin\left(\omega t-\frac{\pi}{2}\right)+\frac{\kappa\bar{\alpha}-4\bar{\alpha}\alpha\omega^2}{2A}\cos\left(2\omega t-\frac{\pi}{2}\right)\notag\\
&-\frac{\bar{\alpha}\tilde{k}\omega}{A}\left(\cos\left(\omega t-\frac{\pi}{2}\right)+\sin\left(2\omega t-\frac{\pi}{2}\right)\right)+\frac{\kappa}{A}g(L)+\frac{\tilde{k}+\kappa t}{A}\tilde{Q}_p\notag\\
&-\frac{\kappa\bar{\alpha}}{A}+\tilde{P},
\end{align}
and assume that the following compatibility conditions are satisfied 
\begin{align}\label{ar28}
u(0,0)&=u(0,L)=0,\notag \\
\eta(0,0)&=\eta_0(0)=g(0), \notag \\
\eta(0,L)&=\eta_0(L)=g(L), \notag \\
\partial_t\eta(0,0)&=\omega\bar{\alpha}+\tilde{Q}_p, \notag \\
\partial_{tt}\eta(0,0)&=-\bar{\alpha}\omega^2, \notag \\
\partial_t\eta(0,L)&=\omega\bar{\alpha}+\tilde{Q}_p,\notag \\
\partial_{tt}\eta(0,L)&=-\bar{\alpha}\omega^2, \notag \\
P(0,0)&=s(0),\notag \\
P(0,L)=s(L)&=\tilde{P}-\frac{\alpha\bar{\alpha}\omega^2+\bar{\alpha}\tilde{k}\omega+\kappa g(L)}{A}.
\end{align}
Then, there exists a time $T<T_0$ such that the problem \eqref{ar24}, \eqref{ar25}, \eqref{ar27} has a local unique solution
\begin{displaymath}
\mathcal{X}(t,z)=(u(t,z), \eta(t,z), P(t,z)),
\end{displaymath}
such that
\begin{displaymath}
\mathcal{X}(t,z)\in \left[C^1((0,T)\times [0,L])\right]^3.
\end{displaymath}
\end{theorem}
\begin{remark}
By using together the regularity assumptions \eqref{ar25} for $u_0(z)$ and $\eta_0(z)$ and \eqref{ar26}, \eqref{ar26b}, we have $P(0,z)=s(z)\in\mathcal{H}^s([0,L])$.
\end{remark}
A natural question which arises now is under which conditions it is possible to get global solutions. In what follows we will show that the answer is yes provided a proper choice of the initial conditions, as is stated in the following theorem on the global existence and uniqueness of a solution to the initial boundary value problem \eqref{ar24}, \eqref{ar25}, \eqref{ar27}.
\begin{theorem}\label{Teorema2}
Let us consider the system
\begin{eqnarray}\label{ar29}
\begin{cases}
\partial_t\eta(t,z)+\partial_t a(t)+u(t,z)-\tilde{Q}_p=0, \\
\alpha\partial_{tt}\eta(t,z)+\tilde{k}\partial_t\eta(t,z)+\kappa\eta(t,z)-AP(t,z)+A\tilde{P}=0,  \\
\rho\partial_t u(t,z)+\rho u(t,z)\partial_z u(t,z)+\partial_z P(t,z)+\beta u(t,z)=0,  
\end{cases}
\end{eqnarray}
where $t\in[0,T]$, $z\in[0,L]$, $a(t)\in C^{\infty}([0,T])$, $\tilde{Q}_p=\displaystyle{\frac{Q_p}{A}}$, $\rho, \beta, \tilde{k}, \kappa, \alpha, A, \tilde{P}\in\mathbb{R}$ are constants, with initial conditions \eqref{ar25} and boundary conditions \eqref{ar27}. Assume that the compatibility conditions \eqref{ar28} are satisfied and that
\begin{equation}\label{ar30}
\lVert f(z)\rVert_{\mathcal{H}^s_z}\leq\frac{\beta}{\rho}, 
\end{equation}
where $f(z)=u(0,z)$.\\
Then there exists a global unique solution 
\begin{displaymath}
\mathcal{X}(t,z)=(u(t,z), \eta(t,z), P(t,z)),
\end{displaymath}
to the problem \eqref{ar29},\eqref{ar25},\eqref{ar27} such that
\begin{displaymath}
\mathcal{X}(t,z)\in \left[C^1((0,T]\times[0,L] )\right]^3,
\end{displaymath}
for every $T\geq0$.
\end{theorem}
The next sections of the paper are devoted to the proof of the Theorems \ref{Teorema} and \ref{Teorema2}.

Moreover the last Section 8 deals with a further development of the global existence theorem. We analyze the particular framework in which we assume initial data very close to periodic conditions that keep the same period of the forcing term $a(t)$ and we state a stability theorem for periodic solutions of the system \eqref{ar1}.

\section{Iterative scheme}\label{iterative}

In this section we start the proof of the Theorem \ref{Teorema}. We divide the proof in two steps. In order to find a solution to the system \eqref{ar24} we define an approximating system by setting up an iterative process and we proceed by induction. In the second step we prove that the iterative scheme converges to a solution of the problem \eqref{ar24} with the initial conditions \eqref{ar25}.\\
We define the approximation of the system \eqref{ar24} as follows; the sequence $\mathcal{X}^{n+1}=(u^{n+1}, \eta^{n+1}, P^{n+1})$ satisfies the following system
\begin{eqnarray}\label{ar32}
\begin{cases}
\partial_t\eta^{n+1}(t,z)+\partial_t a(t)+u^n(t,z)-\tilde{Q}_p=0,\\
\alpha\partial_{tt}\eta^{n+1}(t,z)+\tilde{k}\partial_t\eta^{n+1}(t,z)+\kappa\eta^{n+1}(t,z)-AP^{n+1}(t,z)+A\tilde{P}=0,\\
\rho\partial_t u^{n+1}(t,z)+\rho u^{n}(t,z)\partial_z u^{n+1}(t,z)+\partial_z P^{n+1}(t,z)+\beta u^{n+1}(t,z)=0,
\end{cases}
\end{eqnarray}
where $z\in[0,L]$, $t\in(0, T_0)$, with $T_0\geq T>0$. The initial conditions are given by 
\begin{align}\label{ar33}
u^{n+1}(0,z)&=u_0(z)=f(z)\in\mathcal{H}^{s}([0,L]),\notag\\
\eta^{n+1}(0,z)&=\eta_0(z)=g(z)\in\mathcal{H}^{s}([0,L]),\notag\\
P^{n+1}(0,z)&=P_0(z)=s(z)\notag\\
&=\int_{0}^{z} e^{\frac{\alpha}{\rho}(z-\zeta)}h(\zeta) \, d\zeta + s(0)e^{\frac{\alpha}{\rho}z}\in\mathcal{H}^{s}([0,L]),
\end{align}
and the boundary conditions are the following 
\begin{align}\label{ar34}
u^{n+1}(t,0)&=u(t,L)=0,\notag\\
\eta^{n+1}(t,0)&=g(0)-a(t)+0,3\bar{\alpha}+\tilde{Q}_pt, \notag\\
\eta^{n+1}(t,L)&=g(L)-a(t)+0,3\bar{\alpha}+\tilde{Q}_pt,\notag \\
P^{n+1}(t,0)&=\frac{\bar{\alpha}\alpha\omega^2-\kappa\bar{\alpha}}{A}\sin\left(\omega t-\frac{\pi}{2}\right)+\frac{\kappa\bar{\alpha}-4\bar{\alpha}\alpha\omega^2}{2A}\cos\left(2\omega t-\frac{\pi}{2}\right)\notag\\
&-\frac{\bar{\alpha}\tilde{k}\omega}{A}\left(\cos\left(\omega t-\frac{\pi}{2}\right)+\sin\left(2\omega t-\frac{\pi}{2}\right)\right)+\frac{\kappa}{A}g(0)+\frac{\tilde{k}+\kappa t}{A}\tilde{Q}_p\notag\\
&-\frac{\kappa\bar{\alpha}}{A}+\tilde{P},\notag\\
P^{n+1}(t,L)&=\frac{\bar{\alpha}\alpha\omega^2-\kappa\bar{\alpha}}{A}\sin\left(\omega t-\frac{\pi}{2}\right)+\frac{\kappa\bar{\alpha}-4\bar{\alpha}\alpha\omega^2}{2A}\cos\left(2\omega t-\frac{\pi}{2}\right)\notag\\
&-\frac{\bar{\alpha}\tilde{k}\omega}{A}\left(\cos\left(\omega t-\frac{\pi}{2}\right)+\sin\left(2\omega t-\frac{\pi}{2}\right)\right)+\frac{\kappa}{A}g(L)+\frac{\tilde{k}+\kappa t}{A}\tilde{Q}_p\notag\\
&-\frac{\kappa\bar{\alpha}}{A}+\tilde{P}.
\end{align}
For the initial step we fix 
\begin{align}\label{ar35}
u^0(t,z)&=u_0(z)=f(z),\notag\\
\eta^0(t,z)&=\eta_0(z)=g(z),\notag\\
P^0(t,z)&=P_0(z)=s(z).
\end{align}
In order to prove the existence of a solution to the approximating system \eqref{ar32} we proceed by induction on the iteration number $n$.\\
\\
\textbf{Basic step:} $n=0$. If we fix $n=0$ the system \eqref{ar32} becomes
\begin{eqnarray}\label{ar36}
\begin{cases}
\partial_t\eta^{1}(t,z)+\partial_t a(t)=\tilde{Q}_p-u^0(t,z),\\
\alpha\partial_{tt}\eta^{1}(t,z)+\tilde{k}\partial_t\eta^{1}(t,z)+\kappa\eta^{1}(t,z)-AP^{1}(t,z)+A\tilde{P}=0,\\
\rho\partial_t u^{1}(t,z)+\rho u^{0}(t,z)\partial_z u^{1}(t,z)+\partial_z P^{1}(t,z)+\beta u^{1}(t,z)=0,
\end{cases}
\end{eqnarray}
with the following initial conditions
\begin{align}\label{ar37}
u^1(0,z)&=f(z)\in\mathcal{H}^{s}([0,L]),\notag\\
\eta^1(0,z)&=g(z)\in\mathcal{H}^{s}([0,L]), \notag\\
P^1(0,z)&=s(z)\in\mathcal{H}^{s}([0,L]).
\end{align}
Now, we prove the existence of a solution to the problem \eqref{ar36}, \eqref{ar37}. From the first equation of the system \eqref{ar36} we obtain a unique solution $\eta^1$ of the form
\begin{equation}\label{ar38}
\eta^1(t,z)=g(z)-f(z)t-a(t)+0,3\bar{\alpha}+\tilde{Q}_pt.
\end{equation}
Then, by using the regularity conditions \eqref{ar37} we have 
\begin{equation}\label{ar39}
 \eta^1(t,z)\in C^{\infty}((0,T_0); \mathcal{H}^{s}([0,L])).
\end{equation}
Now we replace \eqref{ar38} into the second equation of \eqref{ar36} and we obtain
\begin{align}\label{ar40}
P^1(t,z)= -\frac{1}{A}\bigg[\alpha\partial_{tt}a(t)-\tilde{k}\tilde{Q}_p&+\tilde{k}f(z)+\tilde{k}\partial_t a(t)-\kappa g(z)+\kappa tf(z)\notag \\
&+\kappa a(t)-0,3\alpha\kappa-\tilde{Q}_p\kappa t -A\tilde{P}\bigg].
\end{align}
Since $a(t)\in C^{\infty}([0,T_0])$ and by using the assumptions \eqref{ar37}, we get that
\begin{equation}\label{ar41}
P^1(t,z)\in C^{\infty}((0,T_0); \mathcal{H}^{s}([0,L])).
\end{equation}
Moreover, since
\begin{equation}
\partial_z P^1(t,z)=-\frac{1}{A}\left[\tilde{k}\partial_z f(z)-\kappa\partial_z g(z)+\kappa t\partial_z f(z)\right],
\end{equation}
we observe that 
\begin{equation}\label{ar42}
\partial_z P^1(z)\in \mathcal{H}^{s-1}([0,L]).
\end{equation}
Finally we need to find $u^1$. The sequence $u^1$ satisfies the following first order hyperbolic system
\begin{eqnarray}\label{ar43}
\begin{cases}
\rho\partial_t u^1(t,z)+\rho f(z)\partial_z u^1(t,z)+\beta u^1(t,z)=-\partial_z P^1(t,z),\\
u^1(0,z)=f(z)\in\mathcal{H}^s([0,L]).
\end{cases}
\end{eqnarray}
In order to find a solution of \eqref{ar43}, we need to recall the following existence theorem for a first order symmetric hyperbolic system (\cite{novot}).
\begin{theorem}\label{iperb}
Let $\mathbb{A}_j(x,t),\, j=0,..., N$ be symmetric $m\times m$ matrices in $M\times (0, T)$ with $M\subset\mathbb{R}^N$, $T>0$, $\mathbf{f}(x, t)$ and $\mathbf{v}_0(x)$, m-dimensional vector functions defined in $M\times (0, T)$ and $M$, respectively. Let us consider the problem 
\begin{subequations}
\begin{equation}\label{ar44}
\mathbb{A}_0(x,t)\partial_t\mathbf{v}+\sum_{j=1}^N\,\mathbb{A}_j(x,t)\partial_j\mathbf{v}+\mathbb{B}(x,t)\mathbf{v}=\mathbf{f}(x,t), 
\end{equation}
\begin{equation}\label{ar45}
\mathbf{v}(x,0)=\mathbf{v}_0(x).
\end{equation}
\end{subequations}
Let
\begin{displaymath}
\mathbb{A}_j\in\left[C([0,T], \mathcal{H}^s(M))\cap C^1((0,T), \mathcal{H}^{s-1}(M))\right]^{m\times m}, \ \ j= 0, 1,..., N,
\end{displaymath}
$\mathbb{A}_0(x,t)$ invertible, $\inf_{x,t}\left\|\mathbb{A}_0(x,t)\right\|_{M\times M}>0$, 
$\mathbb{B}\in C((0,T); \mathcal{H}^{s-1}(M))^{m\times m}$, $\mathbf{f}\in C((0,T); \mathcal{H}^s(M))^m$, $\mathbf{v}_0\in\mathcal{H}^s(M)^m$, where $s>\frac{N}{2}+1$ is an integer. Then there exists a unique solution to \eqref{ar44}, \eqref{ar45}, i.e. a function 
\begin{displaymath}
\mathbf{v}\in \left[C([0,T); \mathcal{H}^s(M))\cap C^1((0,T); \mathcal{H}^{s-1}(M))\right]^m
\end{displaymath}
satisfying \eqref{ar44} and \eqref{ar45} pointwise (i.e. in the classical sense).
\end{theorem}
\par
We apply the Theorem \ref{iperb} to the problem \eqref{ar43}, with $M=[0,L]$, $\mathbf{f}(x, t)=\partial_z P^1(t,z)$ and we observe that we fullfill the hypothesis of the Theorem \ref{iperb}. Then, there exists a unique solution $u^1$ such that
\begin{displaymath}
u^1(t,z)\in C([0,T_0); \mathcal{H}^{s-1}([0,L]))\cap C^1((0,T_0); \mathcal{H}^{s-2}([0,L])).
\end{displaymath}
Finally, we have shown that there exists a unique solution
\begin{equation}\label{ar46}
\mathcal{X}^1(t,z)=(\eta^1(t,z), P^1(t,z), u^1(t,z)),
\end{equation}
to the system \eqref{ar36} with initial data \eqref{ar37} such that 
\begin{align}\label{ar47}
\eta^1(t,z)&\in C^{\infty}((0,T_0); \mathcal{H}^{s}([0,L])), \notag\\
P^1(t,z)&\in C^{\infty}((0,T_0); \mathcal{H}^{s}([0,L])), \notag\\
u^1(t,z)\in C([0,T_0); \mathcal{H}&^{s-1}([0,L]))\cap C^1((0,T_0); \mathcal{H}^{s-2}([0,L])).
\end{align}
Now, we assume that for the \textbf{\textit{n}}\textbf{-th step} there exists a unique solution 
\begin{equation}\label{ar48}
\mathcal{X}^n(t,z)=(\eta^n(t,z), P^n(t,z), u^n(t,z)),
\end{equation}
such that
\begin{align}\label{ar49}
\eta^n(t,z)&\in C^{\infty}((0,T_0); \mathcal{H}^{s}([0,L])), \notag\\
P^n(t,z)&\in C^{\infty}((0,T_0); \mathcal{H}^{s}([0,L])), \notag\\
u^n(t,z)\in C([0,T_0); \mathcal{H}&^{s-1}([0,L]))\cap C^1((0,T_0); \mathcal{H}^{s-2}([0,L])),
\end{align}
and we want to prove the existence of the $(n+1)$-th step.\\
\textbf{\textit{(n+1)}-th step.}
The $n+1$ iteration corresponds to the system
\begin{eqnarray}\label{ar50}
\begin{cases}
\partial_t\eta^{n+1}(t,z)+\partial_t a(t)+u^n(t,z)-\tilde{Q}_p=0,\\
\alpha\partial_{tt}\eta^{n+1}(t,z)+\tilde{k}\partial_t\eta^{n+1}(t,z)+\kappa\eta^{n+1}(t,z)-AP^{n+1}(t,z)+A\tilde{P}=0,\\
\rho\partial_t u^{n+1}(t,z)+\rho u^{n}(t,z)\partial_z u^{n+1}(t,z)+\partial_z P^{n+1}(t,z)+\beta u^{n+1}(t,z)=0,
\end{cases}
\end{eqnarray}
with the initial data \eqref{ar33}.\\
Similarly to the basic step we can write
\begin{equation}\label{ar51}
\eta^{n+1}(t,z)=g(z)-\int_0^t u^n(s,z)\,ds-a(t)+0,3\bar{\alpha}+\tilde{Q}_pt,
\end{equation}
and we get 
\begin{displaymath}
\eta^{n+1}(t,z)\in C^{\infty}((0,T_0); \mathcal{H}^{s}([0,L])).
\end{displaymath}
We replace \eqref{ar51} in the second equation of the system \eqref{ar50} and we obtain
\begin{align}\label{ar52}
P^{n+1}(t,z)=& \frac{1}{A}\bigg[-\alpha\partial_t u^n(t,z)-\alpha\partial_{tt}a(t)+\tilde{k}\tilde{Q}_p-\tilde{k}u^n(t,z)-\tilde{k}\partial_t a(t)\notag\\ 
&+\kappa g(z)-\kappa\int_0^tu^n(s,z)\,ds-\kappa a(t)+0,3\bar{\alpha}\kappa\notag\\
&+\tilde{Q}_p\kappa t+A\tilde{P}\bigg].
\end{align}
From the $n$-th iteration we know there exists a unique solution 
\begin{equation}\label{ar53}
u^n(t,z)\in C([0,T_0); \mathcal{H}^{s-1}([0,L]))\cap C^1((0,T_0); \mathcal{H}^{s-2}([0,L])),
\end{equation} 
by which we get
\begin{equation}\label{ar54}
P^{n+1}(t,z)\in C((0,T_0); \mathcal{H}^{s-2}([0,L])).
\end{equation}
After that, we observe that the last equation of the system \eqref{ar50}
\begin{equation}\label{ar55}
\rho\partial_t u^{n+1}(t,z)+\rho u^n(t,z)\partial_z u^{n+1}(t,z)+\beta u^{n+1}(t,z)= \partial_z P^{n+1}(t,z),
\end{equation}
satisfies the hypothesis of the Theorem \ref{iperb}. In fact
\begin{align}
\partial_z P^{n+1}(t,z)=\frac{1}{A}\bigg[-\alpha\partial_{zt}u^n(t,z)&-\tilde{k}\partial_z u^n(t,z)+\kappa\partial_z g(z)\notag \\
&-\kappa\int_0^t\partial_z u^n(s,z)\,ds\bigg],
\end{align}
and by using \eqref{ar53} we get
\begin{equation}\label{ar56}
\partial_z P^{n+1}(t,z)\in C((0,T_0); \mathcal{H}^{s-3}([0,L])).
\end{equation}
So by applying the Theorem \ref{iperb}, we can assert that there exists a unique solution 
\begin{equation}\label{ar57}
u^{n+1}(t,z)\in C([0,T_0); \mathcal{H}^{s-3}([0,L]))\cap C^1((0,T_0); \mathcal{H}^{s-4}([0,L])).
\end{equation} 
Finally we can summarize the existence of a solution to the approximating system in the following proposition.
\begin{proposition}\label{propos}
There exists a unique triple of solutions 
\begin{equation}\label{ar58}
\mathcal{X}^{n+1}(t,z)=(\eta^{n+1}(t,z), P^{n+1}(t,z), u^{n+1}(t,z)),
\end{equation}
to the problem \eqref{ar50},\eqref{ar33}, where
\begin{align}\label{ar59}
\eta^{n+1}(t,z)&\in C^{\infty}((0,T_0); \mathcal{H}^{s}([0,L])), \notag\\
P^{n+1}(t,z)&\in C((0,T_0); \mathcal{H}^{s-2}([0,L])), \notag\\
u^{n+1}(t,z)\in C([0,T_0); \mathcal{H}&^{s-3}([0,L]))\cap C^1((0,T_0); \mathcal{H}^{s-4}([0,L])).
\end{align}
\end{proposition}

\section{Convergence of the iterative scheme}

In this section we will prove that the iterative scheme previously constructed converges to a solution of the problem \eqref{ar24}, \eqref{ar25}, \eqref{ar27}. \\
First of all we will recover suitable energy estimate for our problem by analyzing the equation $\eqref{ar50}_3$ that is the most tricky one because of its nonlinearity but at the same time it is the equation that allows us to obtain results about the whole system \eqref{ar24}.\\
We prove that $\{u^n\}$ is a Cauchy sequence and so we get the existence of a classical solution to the equation $\eqref{ar24}_3$. Hence, we will be able to prove the strong convergence for the other sequences of unknowns and then the existence of classical solutions to the equations $\eqref{ar24}_1$ and $\eqref{ar24}_2$.

\subsection{Higher order energy estimates for the sequence $\pmb{\{u^n\}}$}\label{5.1}

We proved that 
\begin{displaymath}
u^{n+1}(t,z)\in C([0,T_0); \mathcal{H}^{s-3}([0,L]))\cap C^1((0,T_0); \mathcal{H}^{s-4}([0,L]))
\end{displaymath}
is a solution to the equation
\begin{equation}\label{ar60}
\rho\partial_t u^{n+1}(t,z)+\rho u^{n}(t,z)\partial_z u^{n+1}(t,z)+\partial_z P^{n+1}(t,z)+\beta u^{n+1}(t,z)=0.
\end{equation}
Henceforth we will not denote explicitly the space and time dependence and in order to simplify the notation, we set
\begin{equation}\label{ar61}
u=u^{n}, \ \ v=u^{n+1}-u_0, \ \  P=P^{n+1}, 
\end{equation}
\begin{equation}\label{ar62}
v_{\alpha}=\frac{\partial^{\alpha}}{\partial z^{\alpha}}v, 
\end{equation}
for any $\alpha<s$.\\
First of all we analyze the third equation of the system \eqref{ar24} that, with the new notation, becomes
\begin{subequations}
\begin{equation}\label{ar63}
\rho\partial_t v+\rho u\partial_z v+\beta v= -\partial_z P-\rho u\partial_zf(z)-\beta f(z),
\end{equation}
\begin{equation}\label{ar64}
v(0,z)=0.
\end{equation}
\end{subequations}
If we apply to the equation \eqref{ar63} the operator $\displaystyle{\frac{\partial^{\alpha}}{\partial z^{\alpha}}}$, then we get
\begin{subequations}
\begin{equation}\label{ar65}
\rho\partial_t v_{\alpha}+\rho u\partial_z v_{\alpha}+\beta v_{\alpha}= \frac{\partial^{\alpha}}{\partial z^{\alpha}}(-\partial_z P-\rho u\partial_z f(z)-\beta f(z))-\rho\left(\frac{\partial^{\alpha}}{\partial z^{\alpha}}u\right)\partial_z v,
\end{equation}
\begin{equation}\label{ar66}
v_{\alpha}(0,z)=0.
\end{equation} 
\end{subequations}
Let's define
\begin{equation}\label{ar67}
E^{\alpha}(t):=\frac{1}{2}\int_{0}^L\,\rho v_{\alpha}v_{\alpha}\, dz.
\end{equation}
Applying the time derivative to \eqref{ar67} and by using the equation \eqref{ar65} we get
\begin{align}\label{ar68}
\frac{dE^{\alpha}}{dt}(t):=\int_{0}^L\,\rho \partial_t v_{\alpha}v_{\alpha}&\, dz= -\int_{0}^L\, \bigg[\frac{\partial^{\alpha}}{\partial z^{\alpha}}(\partial_z P+\rho u\partial_z f(z)+\beta f(z))\notag \\
&+\rho u\partial_z v_{\alpha}+\beta v_{\alpha}+\rho\left(\frac{\partial^{\alpha}}{\partial z^{\alpha}}u\right)\partial_z v\bigg] v_{\alpha}\,dz.
\end{align} 
Now, we estimate each one of the terms of the right hand side of \eqref{ar68}.
\begin{equation}\label{ar69}
\left\lvert\int_{0}^L\,\frac{\partial^{\alpha+1}}{\partial z^{\alpha+1}}Pv_{\alpha}\,dz\right\rvert\leq \frac{1}{2}\bigg(\left\| \frac{\partial^{\alpha+1}}{\partial z^{\alpha+1}}P\right\|^2_2+\left\| v_{\alpha}\right\|^2_2\bigg),
\end{equation}
\begin{align}\label{ar70}
\left\lvert\int_{0}^L\,\rho\frac{\partial^{\alpha}}{\partial z^{\alpha}}\left(u\partial_z f(z)\right) v_{\alpha}\,dz\right\rvert&\leq\left\lvert\rho\int_{0}^L\,\left(\frac{\partial^{\alpha}}{\partial z^{\alpha}}u\right)\partial_z f(z) v_{\alpha}\,dz\right\rvert\notag\\
&+\left\lvert\rho\int_{0}^L\,u\left(\frac{\partial^{\alpha}}{\partial z^{\alpha}}\partial_z f(z)\right) v_{\alpha}\,dz\right\rvert\notag\\
&=I_1+I_2,
\end{align}
and
\begin{equation}
I_1\leq\frac{\rho}{2}\left\|v_{\alpha}\right\|_2^2+\frac{\rho}{2}\left\|\partial_z f(z)\right\|_{\infty}^2\left\|\frac{\partial_z^{\alpha}}{\partial z^{\alpha}}u\right\|^2_2,
\end{equation}
\begin{equation}
I_2\leq \frac{\rho}{2}\left\|v_{\alpha}\right\|_2^2+\frac{\rho}{2}\left\|u\right\|_{\infty}^2\left\|\frac{\partial_z^{\alpha+1}}{\partial z^{\alpha+1}}f(z)\right\|_2^2,
\end{equation}
In a similar way for the last terms we get
\begin{align}\label{ar71}
\bigg\lvert\int_{0}^L\,&\left(\beta\frac{\partial^{\alpha}}{\partial z^{\alpha}}f(z)+\rho u\partial_z v_{\alpha}+\beta v_{\alpha}+\rho\bigg(\frac{\partial^{\alpha}}{\partial z^{\alpha}}u\bigg)\partial_z v\right)v_{\alpha}\,dz\bigg\rvert \notag\\
&\leq\frac{\lvert\beta\rvert}{2}\bigg(\left\|\frac{\partial^{\alpha}}{\partial z^{\alpha}}f(z)\right\|^2_2+\left\| v_{\alpha}\right\|_2^2\bigg)+\frac{\rho}{2}\left\|\partial_z u\right\|_{\infty}\left\| v_{\alpha}\right\|_2^2+\lvert\beta\rvert\left\| v_{\alpha}\right\|_2^2\notag\\
&+\rho\left\| v_{\alpha}\right\|_2^2\left\| \frac{\partial^{\alpha}}{\partial z^{\alpha}}u\right\|_2,
\end{align}
Then, by summing up \eqref{ar69}-\eqref{ar71} we get
\begin{equation}\label{ar72}
\frac{dE^{\alpha}}{dt}(t)\leq C+\tilde{C}\left\|v_{\alpha}\right\|_2^2,
\end{equation}
where
\begin{equation}\label{ar73}
\tilde{C}=\sup_{t\in[0,T]}\left\{\frac{1}{2}+\rho+\frac{3\lvert\beta\rvert}{2}+\frac{\rho}{2}\left\|\partial_z u\right\|_{\infty}+\rho\left\|\frac{\partial^{\alpha}}{\partial z^{\alpha}}u\right\|_2\right\},
\end{equation}
\begin{align}\label{ar74}
C=\sup_{t\in[0,T]}&\bigg\{\frac{1}{2}\left\|\frac{\partial^{\alpha+1}}{\partial z^{\alpha+1}}P\right\|_2^2+\frac{\rho}{2}\left\|\frac{\partial^{\alpha}}{\partial z^{\alpha}u}\right\|_2^2\left\|\partial_z f(z)\right\|_{\infty}^2\notag \\
&+\frac{\rho}{2}\left\|u\right\|_{\infty}^2\left\|\frac{\partial^{\alpha+1}}{\partial z^{\alpha+1}}f(z)\right\|_2^2+\frac{\lvert\beta\rvert}{2}\left\|\frac{\partial^{\alpha}}{\partial z^{\alpha}}f(z)\right\|_2^2\bigg\}.
\end{align}
Now, after a detailed analysis of the terms \eqref{ar73} and \eqref{ar74} we have
\begin{equation}\label{ar75}
\frac{d\,E^{\alpha}(t)}{d\,t}\leq C_1\left(1+\left\|v_{\alpha}\right\|^2_2\right),
\end{equation}
where $C_1=C_1\left(\left\|u\right\|^2_{L^{\infty}_t\mathcal{H}^{s-3}_z}\right)$ is a locally bounded function of its argument.
Summing up \eqref{ar75} over $\alpha<s-3$ and integrating over $(0,t)$, where $t\in(0,T)$, $T\leq T_0$, we find
\begin{align}\label{ar76}
\sum_{\alpha<s-3}E^{\alpha}(t)\leq C_1(1+\left\|v\right\|_{L^{\infty}_t\mathcal{H}^{s-3}_z}^2)T.
\end{align}
Taking the $\sup$ over $t\in(0,T)$ in \eqref{ar76} we get
\begin{equation}\label{ar77}
\left\|v\right\|^2_{L^{\infty}_t\mathcal{H}^{s-3}_z}\leq C_1(\left\|u\right\|_{L^{\infty}_t\mathcal{H}^{s-3}_z})(1+\left\|v\right\|^2_{L^{\infty}_t\mathcal{H}^{s-3}_z})T,
\end{equation}
then
\begin{equation}\label{ar78}
\left\|v\right\|^2_{L^{\infty}_t\mathcal{H}^{s-3}_z}\leq\frac{C_1T}{1-C_1T}:=C_2^2T.
\end{equation}
By using the notation \eqref{ar61} this means that
\begin{equation}\label{ar79}
\left\|u^{n+1}-u_0\right\|_{L^{\infty}_t\mathcal{H}^{s-3}_z}=\left\|u^{n+1}-f(z)\right\|_{L^{\infty}_t\mathcal{H}^{s-3}_z}\leq C_2(\left\|u^{n}\right\|_{L^{\infty}_t\mathcal{H}^{s-3}_z})\sqrt{T}.
\end{equation}
We take
\begin{equation}\label{ar80}
r_0>\left\| f(z)\right\|_{\mathcal{H}^s_z}+C_2(\left\| f(z)\right\|_{\mathcal{H}^s_z}),
\end{equation}
and
\begin{equation}\label{ar81}
\sqrt{T}<\left(\sup_{0\leq r\leq r_0}C_2(r)\right)^{-1}(r_0-\left\| f(z)\right\|_{\mathcal{H}^s_z}).
\end{equation}
We can show that for all $n\geq 0$ 
\begin{equation}\label{ar82}
\left\| u^n\right\|_{L^{\infty}_t\mathcal{H}^{s-3}_z}\leq r_0.
\end{equation}
So, assuming \eqref{ar81} true for a given $n\geq 0$, by \eqref{ar79} we have
\begin{equation}\label{ar83}
\left\| u^{n+1}\right\|_{L^{\infty}_t\mathcal{H}^{s-3}_z}\leq \left\| f(z)\right\|_{\mathcal{H}^{s}_z}+\sup_{0\leq r\leq r_0} C_2(r)\sqrt{T}<r_0,
\end{equation}
by induction \eqref{ar82} follows for all $n$. By using  \eqref{ar83} and the equation $\eqref{ar50}_3$ we obtain also that
\begin{align}\label{ar84}
\left\|\partial_t u^{n}\right\|_{L^{\infty}_t\mathcal{H}^{s-3}_z}&\leq \left(\left\| f(z)\right\|_{\mathcal{H}^{s}_z}+\sup_{0\leq r\leq r_0} C_2(r)\sqrt{T}\right)^2\notag\\
&+\frac{1}{\left\lvert\rho\right\rvert}\left(\left\lvert\frac{\kappa}{A}\right\rvert T+\left\lvert\frac{\tilde{k}}{A}\right\rvert+\lvert\beta\rvert\right)\left(\left\| f(z)\right\|_{\mathcal{H}^{s}_z}+\sup_{0\leq r\leq r_0} C_2(r)\sqrt{T}\right) \notag\\
&+\left\lvert\frac{\alpha}{A\rho}\right\rvert\mathcal{C}+\left\lvert\frac{\tilde{k}}{A\rho}\right\rvert\left\| g(z)\right\|_{\mathcal{H}^{s}_z},
\end{align}
where 
\begin{equation}\label{cost}
\mathcal{C}=\mathcal{C}\left(\left\| f(z)\right\|_{\mathcal{H}^{s}_z}, \left\| g(z)\right\|_{\mathcal{H}^{s}_z}\right)
\end{equation}
is the constant inherited by the $n$ iterations.
\subsection{Convergence of the sequence $\pmb{\{u^n\}}$}\label{5.2}

Now we are almost ready to prove the convergence of $\{u^n\}_{n=0}^{\infty}$. Subtracting two subsequent equations in \eqref{ar60} we get
\begin{equation}\label{ar85}
\rho\partial_t(u^{n+1}-u^n)+\rho u^n\partial_z(u^{n+1}-u^n)+\beta(u^{n+1}-u^n)= G_n,
\end{equation}
where
\begin{equation}\label{ar86}
G_n:=\rho(u^{n-1}-u^{n})\partial_z u^n -\partial_z(P^{n+1}-P^n).
\end{equation}
Following exactly the same line of arguments adopted for the energy estimates of the previous section, we get
\begin{equation}\label{ar87}
\left\|u^{n+1}-u^n\right\|_{L^{\infty}_t L^2_z}\leq C_5(\left\|f(z)\right\|_{\mathcal{H}^s_z},T)\sqrt{T}\left\|G_n\right\|_{L^{\infty}_t L^2_z}.
\end{equation}
By \eqref{ar81}, \eqref{ar82}, \eqref{ar86} we get
\begin{equation}\label{ar88}
\left\| G_n\right\|_2\leq\rho\left\|u^n-u^{n-1}\right\|_2\left\|u^n\right\|_{\mathcal{H}^{s-3}_z}+\left\|\partial_z(P^{n+1}-P^n)\right\|_2.
\end{equation}
Taking into account that
\begin{equation}
\left\|\partial_z(P^{n+1}-P^n)\right\|_{L^{\infty}_t L^2_z}\leq\frac{\lvert\alpha\rvert+\lvert\tilde{k}\rvert+\lvert\kappa\rvert T}{\lvert A\rvert}\left\|u^n-u^{n-1}\right\|_{L^{\infty}_t\mathcal{H}^{s-4}_z},
\end{equation}
we have
\begin{equation}\label{ar89}
\left\|G_n\right\|_{L^{\infty}_t L^2_z}
\leq C_6(r_0, T)\left\| u^n-u^{n-1}\right\|_{L^{\infty}_t \mathcal{H}^{s-4}_z}.
\end{equation}
Hence, by using \eqref{ar87} and \eqref{ar89} we obtain
\begin{align}\label{ar90}
\left\| u^{n+1}-u^n\right\|_{L^{\infty}_t L^2_z}\leq C_7(\left\| f(z)\right\|_{\mathcal{H}^s_z},r_0,T)\sqrt{T}\left\| u^n-u^{n-1}\right\|_{L^{\infty}_t \mathcal{H}^{s-4}_z}.
\end{align}
Assuming $T$ so small that
\begin{equation}\label{ar91}
a:= C_7(\left\| f(z)\right\|_{\mathcal{H}^s_z},r_0,T)\sqrt{T}<1
\end{equation}
we find that
\begin{equation}\label{ar92}
\sum_{n=1}^{\infty} \left\| u^{n+1}-u^n\right\|_{L^{\infty}_t L^2_z}\leq\sum_{n=0}^{\infty}a^n\left\| u^{1}-u^0\right\|_{L^{\infty}_t \mathcal{H}^{s-4}_z}<\infty
\end{equation}
which implies that there exists $u\in L^{\infty}\left((0,T); L^2([0,L])\right)$ such that
\begin{equation}\label{ar93}
u^n\rightarrow u \ \ \mbox{strongly in} \ \  L^{\infty}((0,T); L^2([0,L]))  \ \ \mbox{for some} \ \  u.
\end{equation}
Then by \eqref{ar82}, \eqref{ar93} and the interpolation inequality \eqref{ar2}, for $0<s'<s-3$ and $l, m\in\mathbb{N}$ we get
\begin{align}\label{ar94}
\left\| u^l(t)-u^m(t)\right\|_{\mathcal{H}^{s'}_z}&\leq C_{s-3}\left\| u^l(t)-u^m(t)\right\|_{L^2_z}^{1-\frac{s'}{s-3}}\left\| u^l(t)-u^m(t)\right\|^{\frac{s'}{s-3}}_{\mathcal{H}^{s-3}_z}\notag\\
&\leq (2r_0)^{\frac{k'}{k-3}}C_{k-3}\left\| u^l(t)-u^m(t)\right\|_{L^{\infty}_t \mathcal{H}^{s-3}_z}^{1-\frac{s'}{s-3}}.
\end{align}
From \eqref{ar94} it follows that
\begin{equation}\label{ar95}
u^n\rightarrow u  \ \ \mbox{strongly in} \ \ C([0,T]; \mathcal{H}^{s'}([0,L])),
\end{equation}
and by the Sobolev embedding theorem we get
\begin{equation}\label{ar96}
u^n\rightarrow u \ \ \mbox{strongly in}\ \ C([0,T]; C^1([0,L])).
\end{equation}
Now, in order to increase the time regularity of $u$ we analyze the following difference
\begin{displaymath}
\rho\partial_t(u^{n+1}-u^n).
\end{displaymath}
We know that
\begin{align}\label{ar97}
\rho\partial_t(u^{n+1}-u^n)=&-\rho u^n\partial_z(u^{n+1}-u^n)-\beta(u^{n+1}-u^n)-\partial_z\left(P^{n+1}-P^n\right)\notag\\
&-\rho(u^n-u^{n-1})\partial_z u^n.
\end{align}
Hence, taking the sup over $t\in(0,T)$ and by using the previous computations we get
\begin{equation}\label{ar98}
\sup_{0<t<T}\left\|\partial_t(u^{n+1}-u^n)\right\|_{L^2_z}\leq C_8(r_0, T)\sqrt{T}\left\|u^{n+1}-u^n\right\|_{L^{\infty}_t \mathcal{H}^{s-4}_z}.
\end{equation}
Sobolev embedding theorem and \eqref{ar95} yield
\begin{equation}\label{ar99}
\partial_t u^n\rightarrow\partial_t u \ \ \mbox{in} \ \ C((0,T]); C([0,L])).
\end{equation}
By considering together \eqref{ar96} and \eqref{ar99} we can say that $u\in C^1((0,T]\times[0,L])$ is a classical solution of the third equation of our problem.

\subsection{Convergence of the sequences $\pmb{\{\eta^n\}}$ and $\pmb{\{P^n\}}$}

In order to study the convergence of $\{\eta^n\}$ e $\{P^n\}$ we proceed exactly as in Sections \ref{5.1}, \ref{5.2}. We begin subtracting two subsequent equations in \eqref{ar51}:
\begin{equation}\label{ar100}
\eta^{n+1}(t,z)-\eta^n(t,z)= -\int_0^t\,(u^n(s,z)-u^{n-1}(s,z))\,ds,
\end{equation}
by which we get
\begin{equation}\label{ar101}
\left\| \eta^{n+1}-\eta^n\right\|_{L^{\infty}_t L^2_z}\leq a\left\| u^n-u^{n-1}\right\|_{L^{\infty}_t \mathcal{H}^{s-3}_z},
\end{equation}
where $a$ is defined as in \eqref{ar91}. 
By using \eqref{ar92} we find
\begin{equation}\label{ar102}
\sum_{n=1}^{\infty}\left\| \eta^{n+1}-\eta^n\right\|_{L^{\infty}_t L^2_z}\leq \sum_{n=1}^{\infty}a^{n} \left\| u^1-u^0\right\|_{L^{\infty}_t \mathcal{H}^{s-3}_z}<\infty.
\end{equation}
Hence it follows that
\begin{equation}\label{ar103}
\eta^{n}\rightarrow\eta \ \  \mbox{in} \ \  {L^{\infty}((0,T); L^2([0,L]))}.
\end{equation}
By the interpolation inequality \eqref{ar2} we obtain
\begin{equation}\label{ar104}
\eta^{n}\rightarrow \eta \ \ \mbox{in} \ \  C([0,T]; \mathcal{H}^s([0,L])).
\end{equation}
We also know that 
\begin{align}\label{ar105}
\partial_t(\eta^{n+1}(t,z)-\eta^n(t,z))&= -(u^n(t,z)-u^{n-1}(t,z)),\notag\\
\partial_{tt}(\eta^{n+1}(t,z)-\eta^n(t,z))&=\frac{1}{\alpha}\bigg[\tilde{k}(u^{n}-u^{n-1})+\kappa \int_0^t\,(u^n-u^{n-1})\,ds\notag\\
&+A(P^{n+1}-P^n)\bigg],
\end{align}
and that
\begin{align}\label{ar106}
P^{n+1}-P^n=\frac{1}{A}\bigg[\alpha\partial_t(u^n-u^{n-1})&+\tilde{k}(u^n-u^{n-1})\notag\\
&+\kappa\int_0^t(u^n-u^{n-1})\,ds\bigg],
\end{align}
then, taking the sup over $t\in(0,T)$, we obtain
\begin{align}
\left\|P^{n+1}-P^n\right\|_{L^{\infty}_t L^2_z}&\leq a\left\|u^n-u^{n-1}\right\|_{L^{\infty}_t\mathcal{H}^{s-4}_z},\notag\\
\left\|\partial_t(\eta^{n+1}-\eta^n)\right\|_{L^{\infty}_t L^2_z}&\leq \left\|u^n-u^{n-1}\right\|_{L^{\infty}_t L^2_z},\notag\\
\left\|\partial_{tt}(\eta^{n+1}-\eta^n)\right\|_{L^{\infty}_t L^2_z}&\leq a\left\|u^n-u^{n-1}\right\|_{L^{\infty}_t \mathcal{H}^{s-4}_z}.
\end{align}
We can conclude that
\begin{align}
\sum_{n=1}^{\infty}\left\| P^{n+1}-P^n\right\|_{L^{\infty}_t L^2_z}&\leq \sum_{n=1}^{\infty}a^{n}\left\| u^1-u^0\right\|_{L^{\infty}_t\mathcal{H}^{s-4}_z}<\infty, \notag\\
\sum_{n=1}^{\infty}\left\|\partial_t(\eta^{n+1}-\eta^n)\right\|_{L^{\infty}_t L^2_z}&\leq\sum_{n=1}^{\infty}a^{n-1}\left\|u^1-u^0\right\|_{L^{\infty}_t \mathcal{H}^{s-3}_z}<\infty,\notag\\
\sum_{n=1}^{\infty}\left\|\partial_{tt}(\eta^{n+1}-\eta^n)\right\|_{L^{\infty}_t L^2_z}&\leq\sum_{n=1}^{\infty}a^{n}\left\|u^1-u^0\right\|_{L^{\infty}_t \mathcal{H}^{s-4}_z}<\infty,
\end{align}
Hence we get
\begin{align}\label{ar107}
P^{n}&\rightarrow P \ \ \mbox{strongly in} \ \  L^{\infty}((0,T); L^2([0,L])),\notag\\
\partial_t\eta^n&\rightarrow\partial_t\eta\ \ \mbox{strongly in} \ \ L^{\infty}((0,T); L^2([0,L])),\notag\\
\partial_{tt}\eta^n&\rightarrow\partial_{tt}\eta \ \ \mbox{strongly in} \ \ L^{\infty}((0,T); L^2([0,L])).
\end{align}
By using once more \eqref{ar2} and taking into account \eqref{ar107}, we obtain that 
\begin{align}\label{ar108}
P^n&\rightarrow P  \ \ \mbox{strongly in} \ \ C([0,T]; \mathcal{H}^{s'}([0,L])),\notag\\
\partial_t\eta^n&\rightarrow\partial_t\eta \ \ \mbox{strongly in} \ \ C([0,T]; \mathcal{H}^{s'}([0,L])),\notag\\
\partial_{tt}\eta^n&\rightarrow\partial_{tt}\eta \ \ \mbox{strongly in} \ \ C([0,T]; \mathcal{H}^{s'}([0,L])).
\end{align}

\section{Existence and uniqueness}

By using the results of the previous section now we are ready to prove the existence of solutions to the system \eqref{ar24} with initial data \eqref{ar25} and boundary conditions \eqref{ar27}.\\
In fact, by considering \eqref{ar95}, \eqref{ar99}, \eqref{ar107}, \eqref{ar108} we can pass into the limit in the system \eqref{ar24} and we get that $(u, \eta, P)\in [C^1([0,L]\times (0,T])]^3$ is a classical solution of \eqref{ar24}.\\
The last step is to prove the uniqueness. To this aim we consider two solutions $u, v$ of the equation
\begin{displaymath}
\rho\partial_t u(t,z)+\rho u(t,z)\partial_z u(t,z)+\partial_z P(t,z)+\beta u(t,z)=0,
\end{displaymath}
with initial condition given in $\eqref{ar25}_1$. Then $w=u-v$ satisfies
\begin{align}\label{ar109}
\rho\partial_t w+\rho u\partial_z w+w(\beta+\rho\partial_z v)&=0,\notag\\
w(0)&=0.
\end{align}
Multiplying \eqref{ar109} by $w$ and by integrating in space we have
\begin{equation}\label{ar110}
\left\|w(t)\right\|_2^2\leq\tilde{C}\int_0^t\,\left\|w(s)\right\|_2^2\,ds,
\end{equation}
where
\begin{equation}\label{ar111}
\tilde{C}:=\sup_{0<t<T}\bigg\{\frac{\rho}{2}\left\|\partial_z u\right\|_2^2+\rho\left\|\partial_z v\right\|_2^2+\beta\bigg\}.
\end{equation}
By Gronwall inequality and \eqref{ar109} we get $w=0$, hence $u=v$.\\
For $\eta$ and $P$ we can proceed in the same way and we obtain the uniqueness result.\\
This concludes the proof of the Theorem \ref{Teorema}.

\section{Global existence of solutions}

In the previous sections we proved that the existence and the uniqueness of a solution to the problem \eqref{ar24}, \eqref{ar25}, \eqref{ar27} is guaranteed in a time interval $[0,T]$ where $T$ is the local existence time defined in Section \eqref{5.1} as 
\begin{equation}\label{ar145a}
\sqrt{T}<\left(\sup_{0\leq r\leq r_0}C_2(r)\right)^{-1}(r_0-\left\| f(z)\right\|_{\mathcal{H}^s_z}).
\end{equation}
We want to know if it possible to extend the local existence time interval and in particular how much we can push forward the maximal time interval in order to find a global solution to our system. The core of the problem which is related to the study of the global existence of a solution is concentrated on the third equation of \eqref{ar24},
\begin{equation}\label{ar112}
\rho\partial_t u+\rho u\partial_z u+\beta u=-\partial_z P.
\end{equation}
In fact it owns the same behaviour of the Burgers' equation that has a blow up of the solutions at a certain ``breaking'' time.\\
Moreover, by looking at the structure of the system \eqref{ar24}, we observe that both $\eta$ and $P$ lifespan can be obtained by analyzing \eqref{ar112}, in other words, as a first step, we can focus only on the lifespan of the velocity flux $u$ that allows us to have the main framework of our problem.\\

\subsection{Homogeneous equation}

First of all we analyze the global existence for the homogeneous equation,
\begin{equation}\label{ar113}
\partial_t u+u\partial_z u+\frac{\beta}{\rho}u=0.
\end{equation}
with the initial condition 
\begin{equation}\label{ar114}
u(0,z)=u_0(z)=f(z)\in\mathcal{H}^s([0,L]),
\end{equation}
with $\displaystyle{s>\frac{9}{2}}$.\\
To this purpose we define the characteristic curve $\Gamma_{\lambda}$ as
\begin{equation}\label{ar115}
\frac{dz}{dt}=u \ \ \mbox{for} \ \ t\geq 0; \ \ z=\lambda \ \ \mbox{for} \ \ t=0.
\end{equation}
By using \eqref{ar113},\eqref{ar114} we find that along the curve $\Gamma_{\lambda}$
\begin{equation}\label{ar116}
\frac{du}{dt}=-\frac{\beta}{\rho}u \ \ \mbox{for} \ \ t\geq 0; \ \ u=f(\lambda) \ \ \mbox{for} \ \ t=0.
\end{equation}
Hence on $\Gamma_{\lambda}$, by using \eqref{ar115} we have
\begin{equation}\label{ar117}
u(z)=f(\lambda)e^{-\frac{\beta}{\rho}t}; \ \ z=\lambda+f(\lambda)(1-e^{-\frac{\beta}{\rho}t}).
\end{equation}
We consider now the space derivative of $u$, namely we set 
\begin{equation}\label{ar118}
\omega=u_z,
\end{equation}
which by \eqref{ar113}, \eqref{ar115}, \eqref{ar117} satisfies along $\Gamma_{\lambda}$ the following ordinary nonlinear differential equation
\begin{equation}\label{ar119}
\omega'+\frac{\beta}{\rho}\omega+\omega^2=0,
\end{equation}
with initial data
\begin{equation}\label{ar120}
\omega=f'(\lambda) \ \ \mbox{for} \ \ t=0.
\end{equation}
By classical ODE method we solve the Riccati Cauchy problem \eqref{ar119}, \eqref{ar120} and we get that on $\Gamma_{\lambda}$
\begin{equation}\label{ar121}
u_z=\displaystyle{\frac{f'(\lambda)e^{-\frac{\beta}{\rho}t}}{\displaystyle{1+f'(\lambda)\frac{\rho}{\beta}(1-e^{-\frac{\beta}{\rho}t})}}}.
\end{equation}
Hence $u$ is a global solution whenever
\begin{equation}\label{ar122}
1+f'(\lambda)\frac{\rho}{\beta}(1-e^{-\frac{\beta}{\rho}t})\neq 0,
\end{equation}
otherwise we have blow-up at the time $T$ given by
\begin{equation}\label{ar123}
T=\frac{\rho}{\beta}\log\left(\frac{f'(\lambda)}{f'(\lambda)+\frac{\beta}{\rho}}\right).
\end{equation}
By analyzing \eqref{ar123} we have that the solution blows up at the finite time $T> 0$ if and only if
\begin{align}\label{ar124}
\displaystyle{\frac{f'(\lambda)}{f'(\lambda)+\frac{\beta}{\rho}}} & \geq 1 \ \ \mbox{namely} \notag \\
f'(\lambda)&<-\frac{\beta}{\rho}.
\end{align}
Then we can conclude that the solution to \eqref{ar113}, \eqref{ar114} exists globally if and only if
\begin{equation}\label{ar125}
f'(\lambda)\geq -\frac{\beta}{\rho}.
\end{equation}

\subsection{Nonhomogeneous equation}

Now we study the nonhomogeneous equation
\begin{equation}\label{ar126}
\partial_t u+u\partial_z u+\frac{\beta}{\rho}u+\frac{\beta}{\rho}\partial_z P=0,
\end{equation}
with initial condition
\begin{equation}\label{ar127}
u(0,z)=u_0(z)=f(z)\in\mathcal{H}^s([0,L]),
\end{equation}
where $\displaystyle{s>\frac{9}{2}}$.\\
We define the characteristic curve $\Gamma_{\lambda}$ as in \eqref{ar115} and we find that along the characteristic $u$ verifies 
\begin{equation}\label{ar128}
\frac{du}{dt}=-\frac{\beta}{\rho}u-\frac{\beta}{\rho}\partial_z P \ \ \mbox{for} \ \ t\geq 0; \ \ u=f(\lambda) \ \ \mbox{for} \ \ t=0.
\end{equation}
As before we denote by
\begin{equation}\label{ar129}
\omega=u_z,
\end{equation}
which by \eqref{ar126}, \eqref{ar128}, \eqref{ar129} satisfies along $\Gamma_{\lambda}$ the ordinary nonlinear differential Cauchy problem
\begin{eqnarray}\label{ar130}
\begin{cases}
\displaystyle{\omega'+\omega^2+\frac{\beta}{\rho}\omega+\frac{\beta}{\rho}P_{zz}=0}, \\
\omega(0)=f'(\lambda).
\end{cases}
\end{eqnarray}
In order to study the solutions of the nonhomogeneous Riccati equation \eqref{ar130}, we need to find a particular solution that allows us to apply the standard methods for ordinary differential equations. \\
To this aim we look for a particular solution $\bar{\omega}(t)$ of the form
\begin{equation}\label{ar131}
\omega(t)=\frac{a(t)}{b(t)}.
\end{equation}
Hence $a(t)$ and $b(t)$ satisfy the following equation 
\begin{equation}\label{ar132}
a'b-b'a+\frac{\beta}{\rho}ab+\frac{\beta}{\rho}P_{zz}b^2+a^2=0.
\end{equation}
In order to solve \eqref{ar132} and find $a(t), b(t)$ we set the following linear system
\begin{eqnarray}\label{ar133}
\begin{cases}
\displaystyle{a'=-\frac{\beta}{\rho}P_{zz}b} \\
\displaystyle{b'=a+\frac{\beta}{\rho}b},
\end{cases}
\end{eqnarray}
that reduces the order of the equation \eqref{ar132}, and we choose the following initial conditions
\begin{align}\label{ar134}
a(0)&=\omega(0), \notag\\
b(0)&=2.
\end{align}
The system \eqref{ar133} admits a unique solution if and only every element of the system coefficient matrix is bounded, in particular we need that the pressure $P$ is such that
\begin{equation}\label{ar135}
\lVert P_{zz}\rVert_{L^{\infty}_{t,z}}<\infty.
\end{equation}
By the estimates for the local existence of solutions (see Section \ref{5.1}) we know that
\begin{equation}\label{ar142}
\left\lVert u\right\rVert_{L^{\infty}_t\mathcal{H}^{s-3}_z}\leq\left\lVert f(z)\right\rVert_{\mathcal{H}^s_z}+\sup_{0\leq r\leq r_0} C_2(r)\sqrt{T}<r_0
\end{equation}
and 
\begin{align}\label{ar142a}
\left\|\partial_t u\right\|_{L^{\infty}_t\mathcal{H}^{s-3}_z}&\leq \left(\left\| f(z)\right\|_{\mathcal{H}^{s}_z}+\sup_{0\leq r\leq r_0} C_2(r)\sqrt{T}\right)^2\notag\\
&+\frac{1}{\left\lvert\rho\right\rvert}\left(\left\lvert\frac{\kappa}{A}\right\rvert T+\left\lvert\frac{\tilde{k}}{A}\right\rvert+\lvert\beta\rvert\right)\left(\left\| f(z)\right\|_{\mathcal{H}^{s}_z}+\sup_{0\leq r\leq r_0} C_2(r)\sqrt{T}\right) \notag\\
&+\left\lvert\frac{\alpha}{A\rho}\right\rvert\mathcal{C}+\left\lvert\frac{\tilde{k}}{A\rho}\right\rvert\left\| g(z)\right\|_{\mathcal{H}^{s}_z},
\end{align}
where $\mathcal{C}$ is defined in \eqref{cost}.\\
Therefore from the equation \eqref{ar24} we obtain 
\begin{equation}\label{ar136}
\lVert P_{zz}(t,z)\rVert_{L^{\infty}_t \mathcal{H}^{s-4}_z}\leq \lvert\rho\rvert \left\|\partial_t u\right\|_{L^{\infty}_t \mathcal{H}^{s-3}_z}+\lvert\rho\rvert\left\|u\right\|^2_{L^{\infty}_t \mathcal{H}^{s-3}_z}+\lvert\beta\rvert\left\| u\right\|_{\mathcal{H}^{s-3}_z}
\end{equation}
and, by using the relations \eqref{ar142} and \eqref{ar142a}, it yields  
\begin{align}\label{ar144}
\lVert P_{zz}(t,z)\rVert_{L^{\infty}_t\mathcal{H}^{s-4}_z}&\leq 2\lvert\rho\rvert\left(\left\| f(z)\right\|_{\mathcal{H}^s_z}+\sup_{r}C_2(r)\sqrt{\tilde{T}}\right)^2\notag\\
&+\left(\left\lvert\frac{\tilde{k}}{A}\right\rvert+2\lvert\beta\rvert\right)\left(\left\lVert f(z)\right\rVert_{\mathcal{H}^s_z}+\sup_{r}C_2(r)\sqrt{\tilde{T}}\right)\notag\\
&+\frac{\lvert\kappa\rvert}{A}\tilde{T}\left(\left\lVert f(z)\right\rVert_{\mathcal{H}^s_z}+\sup_{r}C_2(r)\sqrt{\tilde{T}}\right)+\left\lvert\frac{\kappa}{A}\right\rvert\lVert g(z)\rVert_{\mathcal{H}^{s}_z}\notag\\
&+\left\lvert\frac{\kappa}{A}\right\rvert\mathcal{C},
\end{align}
where $\tilde{T}\in(0,T)$ and $T$ is the local existence time defined in \eqref{ar145a}. Hence we can conclude that the condition \eqref{ar135} is satisfied.\\
In order to find the solution of \eqref{ar133}, first of all we compute the coefficient matrix eigenvalues of the linear system and we get
\begin{equation}\label{ar140}
\lambda_{1,2}={\displaystyle{\frac{\frac{\beta}{\rho}\pm\sqrt{\left(\frac{\beta}{\rho}\right)^2-4\frac{\beta}{\rho}P_{zz}}}{2}}}.
\end{equation}
Since the model we are dealing with is associated to a physiological flux, we look for real solutions $\omega(t)$. Hence, in order to obtain real eigenvalues \eqref{ar140} we need that $P_{zz}$ verifies the following costraint
\begin{equation}\label{ar141a}
\sup_{t,z}\,\lvert P_{zz}\rvert\leq\frac{\beta}{4\rho}.
\end{equation}
The condition \eqref{ar141a} is fullfilled if we are able to show that
\begin{equation}\label{ar141}
\lVert P_{zz}(t,z)\rVert_{L^{\infty}_t\mathcal{H}^{s-4}_z}\leq\frac{\beta}{4\rho},
\end{equation}
which is achieved by using again the estimate \eqref{ar144}.\\
At this point our goal is to prove that there exists a time $\tilde{T}$ such that the condition \eqref{ar141} is satisfied on the interval $[0, \tilde{T}]$. We can show that this is true if and only if the following condition is satisfied, 
\begin{equation}\label{ar143}
\left\lVert f(z)\right\rVert_{\mathcal{H}^s_z}\leq\frac{\beta}{\rho}.
\end{equation}
In fact, by applying the condition \eqref{ar143} to \eqref{ar144} we get 
\begin{align}\label{ar145}
\lVert P_{zz}(t,z)\rVert_{L^{\infty}_t\mathcal{H}^{s-4}_z}&\leq 2\lvert\rho\rvert\left(\frac{\beta^2}{\rho^2}+\tilde{T}\left(\sup_{r}C_2(r)\right)^2+\left\lvert\frac{\beta}{\rho}\right\rvert\sqrt{\tilde{T}}\sup_{r}C_2(r)\right)\notag\\
&+\left(\left\lvert\frac{\tilde{k}}{A}\right\rvert+2\lvert\beta\rvert+\left\lvert\frac{\kappa}{A}\right\rvert\tilde{T}\right)\left(\left\lvert\frac{\beta}{\rho}\right\rvert+\sup_{r}C_2(r)\sqrt{\tilde{T}}\right)\notag\\
&+\left\lvert\frac{\alpha}{A}\right\rvert\mathcal{C}+\left\lvert\frac{\kappa}{A}\right\rvert\lVert g(z)\rVert_{\mathcal{H}^s_z},
\end{align}
which actually fullfills \eqref{ar141} for a certain time $\tilde{T}\leq T$.\\ 
With this last result we have proved the existence of a real particular solution of the system \eqref{ar130}, that is defined in the following way
\begin{equation}\label{ar138}
\bar{\omega}(t)=\frac{a(t)}{b(t)}=-{\displaystyle{\frac{c_1e^{\lambda_1t}+c_2e^{\lambda_2t}}{c_1\frac{\rho\lambda_1}{\beta P_{zz}}e^{\lambda_1t}+c_2\frac{\rho\lambda_2}{\beta P_{zz}}e^{\lambda_2t}}}},
\end{equation}
if $\lambda_1\neq\lambda_2$, where 
\begin{align}\label{ar139}
c_1&=f'(\lambda)-{\displaystyle{\frac{2\beta P_{zz}-\rho f'(\lambda)\lambda_1}{\rho(\lambda_1-\lambda_2)}}},\notag\\
c_2&={\displaystyle{\frac{2\beta P_{zz}-\rho f'(\lambda)\lambda_1}{\rho(\lambda_1-\lambda_2)}}},
\end{align}
while 
\begin{equation}\label{ar138bis}
\bar{\omega}(t)=\frac{a(t)}{b(t)}={\displaystyle{\frac{-\frac{\beta}{2\rho}c_1+c_2\left(-\frac{\beta}{2\rho}t+1\right)}{c_1+c_2t}}},
\end{equation}
if $\lambda_1=\lambda_2$, where 
\begin{align}\label{ar139bis}
c_1&=2,\notag\\
c_2&=f'(\lambda)+\frac{\beta}{\rho}.
\end{align}
Therefore we can say that the initial value problem \eqref{ar130} admits the \,following unique solution
\begin{equation}\label{ar147}
\omega(t)=\bar{\omega}(t)+\frac{1}{y(t)},
\end{equation}
where
\begin{equation}\label{ar148}
y(t)=\displaystyle{e^{\int_0^t\left(2\bar{\omega}(\tau)+\frac{\beta}{\rho}\right)\,d\tau}}\left\{\frac{2}{f'(\lambda)}+\int_0^t e^{-\int_0^t\left(2\bar{\omega}(s)-\frac{\beta}{\rho}\right)\,ds}\,d\tau\right\},
\end{equation}
that is well defined for every $t\in[0,\tilde{T}]$.

\subsection{Global existence time}

Summarizing, we have proved that under the condition \eqref{ar141} there exists a solution $u(t)$ defined for any time $t\in[0,T^{\ast}]$, such that 
\begin{displaymath}
u(t) \in C\left([0, T^{\ast}]; \mathcal{H}^{s-3}([0, L])\right)\cap C^1\left([0, T^{\ast}); \mathcal{H}^{s-4}([0,L])\right),
\end{displaymath}
with 
\begin{equation}\label{ar150}
T^{\ast}\leq\min\left\{\tilde{T}; T\right\},
\end{equation}
where $T$ is the local existence time defined in \eqref{ar145a}.\\
At this point the main question is wheter it is possible to extend the time $T^{\ast}$ in order to obtain the global existence of the solution $u(t,z)$. For this purpose we need the following Sharp Continuation Principle (see Majda (\cite{majda}), Thm 2.2, p.\,46).
\begin{theorem}[A Sharp Continuation Principle]\label{TeoremaG}
Let be $u_0\in\mathcal{H}^s$ for some $s>\frac{9}{2}$ and $T'>0$ some given time. Assume that for any interval of classical existence $[0, T^{\ast}]$, $T^{\ast}\leq T'$ for u(t) solution of the equation \eqref{ar126}, the following a priori estimate is satisfied:
\begin{equation}\label{ar152}
\left\lVert u_z\right\rVert_{L^{\infty}_t}\leq M, \ \ \ \ \ \ \ \ 0\leq t\leq T^{\ast},
\end{equation}
where $M$ is a fixed constant independent of $T^{\ast}$. \\
Then the classical solution $u(t)$ exists on the interval $[0, T']$, with $u(t)$ in $C\left([0,T'], \mathcal{H}^{s-3}([0, L])\right)\cap C^1\left([0,T'], \mathcal{H}^{s-4}([0,L])\right)$. Furthermore, $u(t)$ satisfies the a priori estimate
\begin{equation}\label{ar153}
\left\| u\right\|_{L^{\infty}_t\mathcal{H}^{s-3}_z}\leq C e^{MT^{\ast}}\left\| f(z)\right\|_{\mathcal{H}^s_z},
\end{equation}
for $T^{\ast}$ with $0\leq T^{\ast}\leq T'$ and the constants $C, M$ depend on the space interval $[0, L]$ and on some physics quantities.
\end{theorem}
By taking into account that 
\begin{equation}
u_z=\omega
\end{equation}
we want to show that the solution \eqref{ar147} satisfies the condition \eqref{ar152} of the previous theorem. We observe that  
\begin{equation}\label{ar154}
\left\lVert u_z\right\rVert_{L^{\infty}_t\mathcal{H}^{s-4}_z}\leq 1+\frac{\beta\left\| P_{zz}\right\|_{L^{\infty}_t\mathcal{H}^{s-5}_z}}{\rho\lambda_1}.
\end{equation} 
Since $\lambda_1\geq\frac{\beta}{\rho}$ and by applying the condition \eqref{ar141} to \eqref{ar154} we obtain 
\begin{align}
\left\lVert u_z\right\rVert_{L^{\infty}_t\mathcal{H}^{s-4}_z}&\leq 1+\left\| P_{zz}\right\|_{L^{\infty}_t\mathcal{H}^{s-5}_z}\notag\\
&\leq 1+\frac{\beta}{4\rho}=M.
\end{align}
Hence, we can conclude that the hypothesis of the Theorem \ref{TeoremaG} are totally fullfilled and the global existence of the solution $u(t,z)$ is proved, as well, provided the condition \eqref{ar143} holds on the initial datum.\\
Moreover, since $\eta(t,z)$ and $P(t,z)$ are obtained by plugging in the first two equations of the system \eqref{ar24} the solution $u(t,z)$ and since they inherit the life span of the latter, the Theorem \ref{TeoremaG} proves actually the global existence of a unique triple solution $\mathcal{X}(t,z)=(\eta, P, u)(t,z)$ to the problem \eqref{ar24}, \eqref{ar25}, \eqref{ar27}.

\section{Existence and stability of a periodic solution}

This last section is devoted to prove the stability for some particular periodic solution for our system. We denote by $\bar{\eta}, \bar{u}, \bar{P}$ the periodic solutions, independent of the spatial variable $z$, namely they satisfy the following system 
\begin{eqnarray}\label{ar155}
\begin{cases}
\partial_t\bar{\eta}(t)+\partial_t a(t)-\tilde{Q}_p=0, \\
\alpha\partial_{tt}\bar{\eta}(t)+\tilde{k}\partial_t\bar{\eta}(t)+\kappa\bar{\eta}(t)-A\bar{P}(t)+A\tilde{P}=0,  \\
\rho\partial_t\bar{u}(t)+\beta\bar{u}(t)=0. 
\end{cases}
\end{eqnarray}

We recall that the forcing function $a(t)$ is given by
\begin{displaymath}
a(t)=\bar{\alpha}\left(1.3+\sin\left(\omega t-\frac{\pi}{2}\right)-\frac{1}{2}\cos\left(2\omega t-\frac{\pi}{2}\right)\right).
\end{displaymath}
Moreover the existence of the periodic solutions is proved for a particular choice of the cerebrospinal fluid production rate $Q_p$.\\
We assume that
\begin{align}\label{ar156}
\tilde{Q}_p=&\,F'(t),\notag \\
F(0)=&\,0, 
\end{align}
where $F(t)$ is a periodic function. For the sake of simplicity and without loss of generality we choose for $F(t)$ the same period of the forcing function $a(t)$, $\bar{T}=\frac{2\pi}{\omega}$. This assumption is allowed by the dynamic we are analyzing, in fact the production rate is well approximated by a periodic function that ensures a constant rate in a proper range of time.\\

\subsection{Existence of a periodic solution}

In order to find a periodic solution $\bar{\chi}(t)=(\bar{\eta}, \bar{u}, \bar{P})(t)$ to the problem \eqref{ar155}, we define the following initial conditions 
\begin{align}\label{ar157}
\bar{\eta}(0)&=\bar{\eta}_0=-\frac{2\pi}{\omega},\notag\\
\bar{u}(0)&=\bar{u}_0=0, \notag\\
\bar{P}(0)&=\bar{P}_0=\frac{\tilde{k}\bar{\alpha}\omega-\bar{\alpha}\alpha\omega^2}{A}-\frac{2\pi\kappa}{A\omega}+\tilde{P}.
\end{align}
Hence we can compute explicitly the solution to the problem \eqref{ar155}, \eqref{ar157},
\begin{align}\label{ar158}
\bar{\eta}(t)&=-\frac{2\pi}{\omega}-a(t)+0,3\bar{\alpha}+F(t), \notag \\
\bar{u}(t)&=0, \notag \\
\bar{P}(t)&=\frac{\bar{\alpha}\alpha\omega^2-\kappa\bar{\alpha}}{A}\sin\left(\omega t-\frac{\pi}{2}\right)-\frac{4\bar{\alpha}\alpha\omega^2-\kappa\bar{\alpha}}{2A}\cos\left(2\omega t-\frac{\pi}{2}\right)\notag\\
&-\frac{\tilde{k}\bar{\alpha}\omega}{A}\left(\cos\left(\omega t-\frac{\pi}{2}\right)+\sin\left(2\omega t-\frac{\pi}{2}\right)\right)+\frac{\kappa}{A}F(t)+\frac{\tilde{k}}{A}F'(t)\notag\\
&+\frac{\alpha}{A}F''(t)-\frac{2\pi\kappa}{A\omega}-\frac{\bar{\alpha}\kappa}{A}+\tilde{P},
\end{align}
where both $\bar{\eta}(t), \bar{P}(t)$ are $C^{\infty}$ functions on $[0,T]$.
The solution \eqref{ar158} is the unique periodic solution whose period is influenced by the forcing function $a(t)$ and it corresponds to $\bar{T}=\frac{2\pi}{\omega}$.\\

\subsection{Stability of the periodic solution}

Now we study the behaviour of the solution $\mathcal{X}(t,z)$ to the problem \eqref{ar24}, \eqref{ar25} when we assume initial data very close to the periodic initial conditions \eqref{ar157}. In order to analyze the stability of the periodic solution it is important to observe that we can push the time $T$ up to the period $\bar{T}=\frac{2\pi}{\omega}$ of the solution $\bar{\mathcal{X}}(t,z)=\left(\bar{\eta}, \bar{P}, \bar{u}\right)(t)$ and we can apply the following theorem on every time interval $[h\bar{T}, (h+1)\bar{T}]$ with $h\in\mathbb{N}$. This explains why the estimates in the proof below are bounded no matter how big we choose the time, in fact by using Thm. \ref{TeoremaG}\, and the periodicity of the system \eqref{ar155} we always obtain the following main result. 
 
\begin{theorem}
Let $T>0$ and $s>\frac{9}{2}$. Let $(\eta, P, u)(t,z)$ and $(\bar{\eta}, \bar{P}, \bar{u})(t)$ be the solutions of the problem \eqref{ar24},\eqref{ar25} and \eqref{ar155},\eqref{ar158} respectively. There exist constants $\delta\in(0,1), K(\delta)$ such that, if 
\begin{equation}\label{ar159}
\left\lVert \eta_0-\bar{\eta}_0\right\rVert_{\mathcal{H}^{s-4}_z}+\left\lVert P_0-\bar{P}_0\right\rVert_{\mathcal{H}^{s-4}_z}+\left\lVert u_0-\bar{u}_0\right\rVert_{\mathcal{H}^{s-4}_z}\leq \delta,
\end{equation}
for all $\delta$, then
\begin{align}\label{ar160}
\sup_{0\leq t\leq T}\bigg(\left\lVert \eta(t,z)-\bar{\eta}(t)\right\rVert_{\mathcal{H}^{s-4}_z}&+\left\lVert P(t,z)-\bar{P}(t)\right\rVert_{\mathcal{H}^{s-4}_z}\notag \\
&+\left\lVert u(t,z)-\bar{u}(t)\right\rVert_{\mathcal{H}^{s-4}_z}\bigg)\leq K(\delta)
\end{align}
for all $\delta$.
\end{theorem}
\begin{proof}
We start by setting 
\begin{align}\label{ar161}
N(t,z)&=\eta(t,z)-\bar{\eta}(t), \notag \\
U(t,z)&=u(t,z)-\bar{u}(t), \notag \\
\Pi(t,z)&=P(t,z)-\bar{P}(t),
\end{align}
which by standard computations satisfy the following system
\begin{eqnarray}\label{ar162}
\begin{cases}
\partial_t N(t,z)+U(t,z)=0, \\
\alpha\partial_{tt}N(t,z)+\tilde{k}\partial_t N(t,z)+\kappa N(t,z)-A\Pi(t,z)=0,  \\
\rho\partial_t U(t,z)+\rho U(t,z)\partial_z U(t,z)+\partial_z\Pi(t,z)+\beta U(t,z)=0.
\end{cases}
\end{eqnarray}
Moreover, the previous system is complemented with the following initial data
\begin{align}\label{ar163}
N(0,z)&=N_0(z)=g(z)+\frac{2\pi}{\omega}, \notag \\
U(0,z)&=U_0(z)=f(z), \notag \\
\Pi(0,z)&=\Pi_0(z)=s(z)-\frac{\tilde{k}\bar{\alpha}\omega-\bar{\alpha}\alpha\omega^2}{A}+\frac{2\pi\kappa}{A\omega}-\tilde{P},
\end{align}
and the following boundary conditions
\begin{align}\label{ar164}
N(t,0)&=g(0)+\frac{A}{\kappa}\tilde{P}+\frac{2\pi}{\omega}, \notag \\
U(t,0)&=0, \notag \\
\Pi(t,0)&=\frac{\kappa}{A}g(0)+\frac{2\pi\kappa}{A\omega}, \notag \\
N(t,L)&=g(L)+\frac{2\pi}{\omega}, \notag \\
U(t,L)&=0, \notag \\
\Pi(t,L)&=\frac{\kappa}{A}g(L)+\frac{2\pi\kappa}{A\omega}.
\end{align}
From now on for the sake of simplicity we will not denote explicitly the time and space dependence.\\
We observe that the assumptions \eqref{ar159} totally fullfills the condition \eqref{ar30}, which means that we can apply all the classical solutions properties in what follows.\\
In order to prove our statement we perform, as before, higher order energy estimates of $(N, \Pi, U)$. From now on, for any $\alpha<s-4$ we will use the following notation $N_{\alpha}=\frac{\partial^{\alpha}}{\partial z^{\alpha}}N$, $\Pi_{\alpha}=\frac{\partial^{\alpha}}{\partial z^{\alpha}}\Pi$ and $U_{\alpha}=\frac{\partial^{\alpha}}{\partial z^{\alpha}}U$. \\
We begin by taking the order $\alpha$ derivative of the third equation of \eqref{ar162}, that becomes
\begin{equation}\label{ar165}
\rho\partial_tU_{\alpha}+\rho U_{\alpha}\partial_zU+\rho U\partial_zU_{\alpha}+\beta U_{\alpha}+\partial_z\Pi_{\alpha}=0,
\end{equation}
in which we plug
\begin{equation}\label{ar166}
\partial_z\Pi_{\alpha}=-\frac{1}{A}\left(\alpha\partial_z\partial_tU_{\alpha}+\tilde{k}\partial_zU_{\alpha}-\kappa\partial_z^{\alpha+1}N_0+\kappa\int_0^t\partial_zU_{\alpha}\,ds\right), 
\end{equation}
obtained by combining the first two equations of \eqref{ar162}. \\
We define the energy as
\begin{equation}\label{ar167}
E^{\alpha}(t):=\frac{1}{2}\int_0^L \rho U_{\alpha}U_{\alpha} dz,
\end{equation}
and applying the time derivative to \eqref{ar167} we get the following estimate
\begin{align}\label{ar168}
\frac{d E^{\alpha}(t)}{dt}&=\rho\int_0^L \partial_t U_{\alpha}U_{\alpha}\,dz\notag \\
&=-\int_0^L \bigg[\rho U_{\alpha}\partial_z U+\rho U\partial_z U_{\alpha}+\beta U_{\alpha}+\frac{1}{A}\bigg(\alpha\partial_z\partial_tU_{\alpha}+\tilde{k}\partial_zU_{\alpha}\notag\\
&-\kappa\partial_z^{\alpha+1}N_0+\kappa\int_0^t\partial_zU_{\alpha}\,ds\bigg)\bigg]U_{\alpha}\,dz.
\end{align}
By estimating and summing up over $\alpha<s-4$ every term in \eqref{ar168}, we get the following estimate
\begin{align}\label{ar169}
\frac{d}{dt}\left\lVert U\right\rVert_{L^{\infty}_{t}\mathcal{H}^{s-4}_z}^{2}&\leq \bar{C}\left\lVert U\right\rVert_{L^{\infty}_{t}\mathcal{H}^{s-4}_z}+\frac{\lvert\kappa\rvert}{2A}\left\lVert N_{0}\right\rVert_{\mathcal{H}^{s-4}_z},
\end{align}
where
\begin{equation}\label{ar170} 
\bar{C}=\sup_{0\leq t\leq T} \left(\frac{3\rho}{2}\left\lVert\partial_z U\right\rVert_{L^{\infty}_{z}}+\lvert\beta\rvert+\frac{\lvert\alpha\rvert+\lvert\tilde{k}\rvert L+\lvert\kappa\rvert(1+A)}{A}\right).
\end{equation}
We apply the Gronwall lemma to \eqref{ar169}, and by using the assumption \eqref{ar159} we obtain 
\begin{align}\label{ar171}
\left\lVert U\right\rVert_{L^{\infty}_{t}\mathcal{H}^{s-4}_z}^{2}&\leq\left(\left\lVert U_{0}\right\rVert_{\mathcal{H}^{s-4}_z}^{2}+\int_{0}^{t}\frac{\lvert\kappa\rvert}{2A}\left\lVert N_{0}\right\rVert_{\mathcal{H}^{s-4}_z}^{2}e^{-\int_{0}^{s}\bar{C}dr}\right)e^{\bar{C}t}\notag\\
&\leq\delta^{2}\left(1+\frac{\lvert\kappa\rvert}{2A}\right)e^{\bar{C}t}.
 \end{align}
Taking into account that $Q_p$, the production rate \eqref{ar156}, is a totally bounded function with period $\bar{T}$, we can easily observe that \eqref{ar171} yields
\begin{align}\label{ar172}
\left\lVert N\right\rVert_{L^{\infty}_{t}\mathcal{H}^{s-4}_z}&\leq\left\lVert N_{0}\right\rVert_{\mathcal{H}^{s-4}_z}+T\left\lVert U\right\rVert_{L^{\infty}_{t}\mathcal{H}^{s-4}_z}\notag\\
&\leq \delta\left(1+T\sqrt{1+\frac{\lvert\kappa\rvert}{2A}}e^{\frac{\bar{C}}{2}t}\right)
\end{align}
and
\begin{align}\label{ar173}
\left\lVert \Pi\right\rVert_{L^{\infty}_{t}\mathcal{H}^{s-4}_z}&\leq \frac{\lvert\alpha\rvert}{A}\left\lVert \partial_{tt}N\right\rVert_{L^{\infty}_{t}\mathcal{H}^{s-4}_z}+\frac{\lvert\tilde{k}\rvert}{A}\left\lVert \partial_{t}N\right\rVert_{L^{\infty}_{t}\mathcal{H}^{s-4}_z}\notag\\
&+\frac{\lvert\kappa\rvert}{A}\left\lVert N\right\rVert_{L^{\infty}_{t}\mathcal{H}^{s-4}_z}+\lvert \tilde{P}\rvert\notag\\
&\leq\delta\left(1+\frac{\lvert\kappa\rvert}{A}+3\bar{\gamma}T\sqrt{1+\frac{\lvert\kappa\rvert}{2A}}\right)e^{\frac{\bar{C}}{2}t}+\lvert\tilde{P}\rvert,
\end{align}
where 
\begin{equation}
\bar{\gamma}=\max\left\{\lvert\kappa\rvert, \lvert\tilde{k}\rvert\frac{\bar{C}}{2}, \lvert\alpha\rvert\frac{\bar{C}^{2}}{4} \right\}.
\end{equation}
Hence, by using \eqref{ar171}, \eqref{ar172} and \eqref{ar173} we finally get
\begin{align}\label{ar174}
\sup_{0\leq t\leq T} \bigg(\left\lVert \eta(t,z)-\bar{\eta}(t)\right\rVert_{\mathcal{H}^{s-4}_z}&+\left\lVert P(t,z)-\bar{P}(t)\right\rVert_{\mathcal{H}^{s-4}_z}+\left\lVert u(t,z)-\bar{u}(t)\right\rVert_{\mathcal{H}^{s-4}_z}\bigg)\notag\\
&\leq \delta\sqrt{1+\frac{\lvert\kappa\rvert}{2A}}e^{\frac{\bar{C}}{2}t}\left(5+T+3\bar{\gamma}T\right)\notag\\
&\leq K(\delta),
\end{align}
which concludes the proof.

\end{proof}

\section{Numerical simulations}

This section is devoted to the validation of the main results obtained in this paper by performing numerical simulations. The computational parameters are based on magnetic resonance measurements. At locations where measured data are not available, lengths and diameters are estimated by combining literature data with measured and computed flows. All the physical constants that appears in our cerebrospinal fluid model are shown in Table \ref{tab1} but, first of all, we have to remember that, for the sake of simplicity, we collected them in the following way:
\begin{align}
\tilde{Q_p}&=\displaystyle{\frac{Q_p}{A}}, \notag \\
\alpha&=\rho A\delta,\notag \\
\tilde{k}&=k_d, \notag \\
\kappa&=k_e, \notag\\
\beta&=\frac{8\mu}{r^2}.
\end{align}

\begin{table}[h]
\centering
\label{tab1}
\begin{tabular}{|c|c|}
\hline
Property&Value\\ 
\hline
Fluid density, $\rho$&$1004-1007$ kg/m$^3$\\ 
\hline
Tissue width, $\delta$&$\sim 5\times 10^{-4}$ m\\
\hline
 Fluid viscosity, $\mu$&$10^{-3}$ Pa s\\ 
 \hline
 Spring elasticity, $k_e$&$8$ N/m\\ 
 \hline
 Brain dampening, $k_d$&$0,35\times 10^{-3}$ (N s)/m\\ 
 \hline
\end{tabular}
 \caption{Tissue and fluid properties.}
\end{table}

Moreover we know that the choroid plexus produces CSF at a rate, $Q_p$, of $0,32$ cm$^3$/min \cite{biblio1} and from medical literature \cite{biblio2}, \cite{biblio1} and published data about ventricular pulsation \cite{biblio3}, \cite{biblio4} we deduce that the amplitude of choroid expansion, $a(t)$, is in the range of $1,29-1,55$ cm$^3$/min at each cardiac cycle.\\
In the inputs list for the simulations an important role is also played by the brain tissue pressure, $\tilde{P}$, that is assumed to be at venous pressure levels ($\sim 10$ mmHg).\\
As we anticipated in section \ref{paragr1}, in this model we assume that the cross sectional area, $A$, is affected only negligibly by the pressure variation and represents for us a constant that we choose in the range of $\sim 3-4$ mm$^2$ according to the clinical data and the compartment section.

Since the problem has an axisymmetric structure, in the numerical approach the model geometry is created in two dimensions. A first order upwind scheme with centered differences is employed for the spatial discretization of the governing equations and for the temporal discretization we adopt a forward Euler scheme with a time step size of $5\times 10^{-3}$.
\\
The boundary conditions involved in numerical simulations are selected by taking into account the conditions \eqref{ar27} stated in Theorem \ref{Teorema}.\\
For the initial data we proceed in the following way: the pressure at the initial time, $P_0(z)$, is implemented exactly as required by the conditions \eqref{ar26} and \eqref{ar26b}, while, for the velocity flux, $u$, and the tissue displacement, $\eta$, we choose 
\begin{equation}\label{numer1}
u_0(z)=4\sin(\pi z)
\end{equation}
which satisfy the condition condition \eqref{ar30}.
\vspace{10pt}

\begin{figure}[h]
\begin{minipage}{0.5\textwidth}
\includegraphics[width=\textwidth]{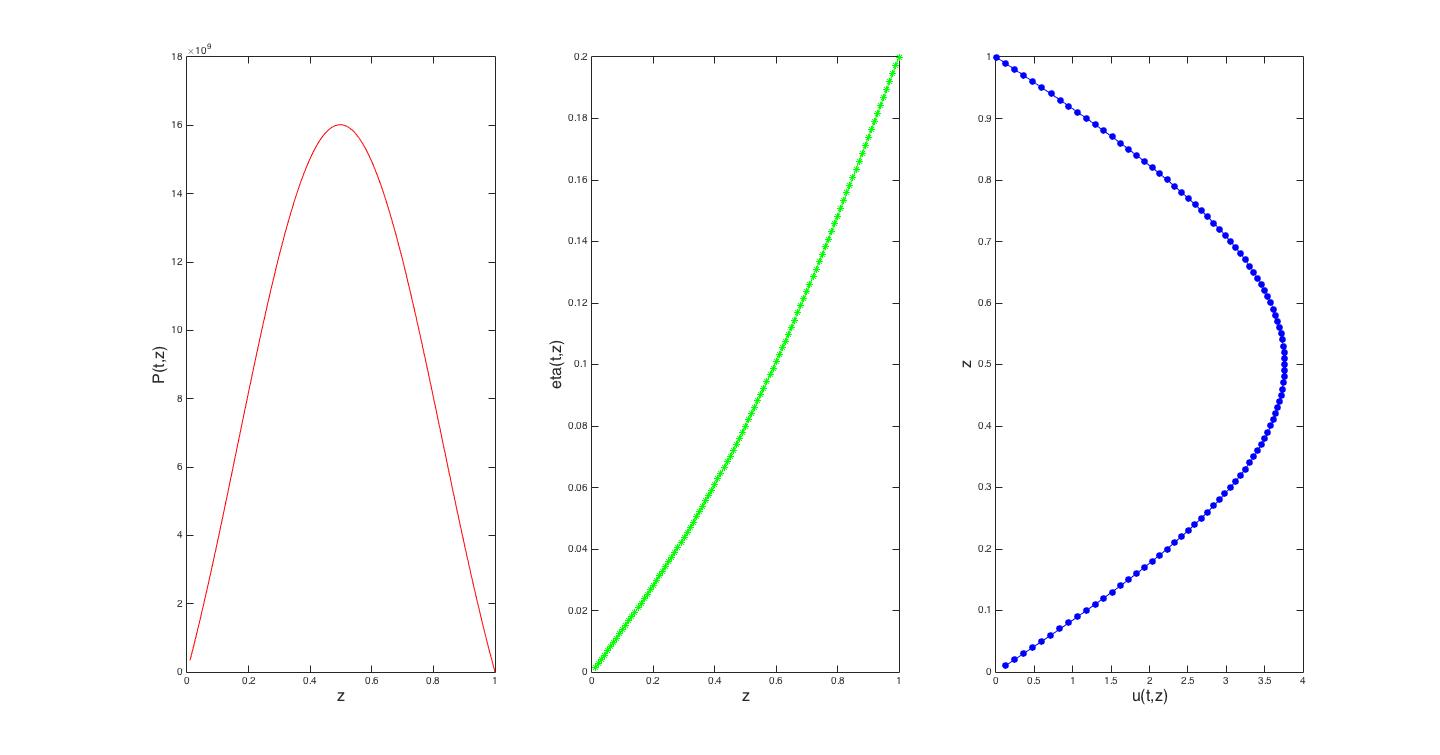}
\end{minipage}
\begin{minipage}{0.5\textwidth}
\includegraphics[width=\textwidth]{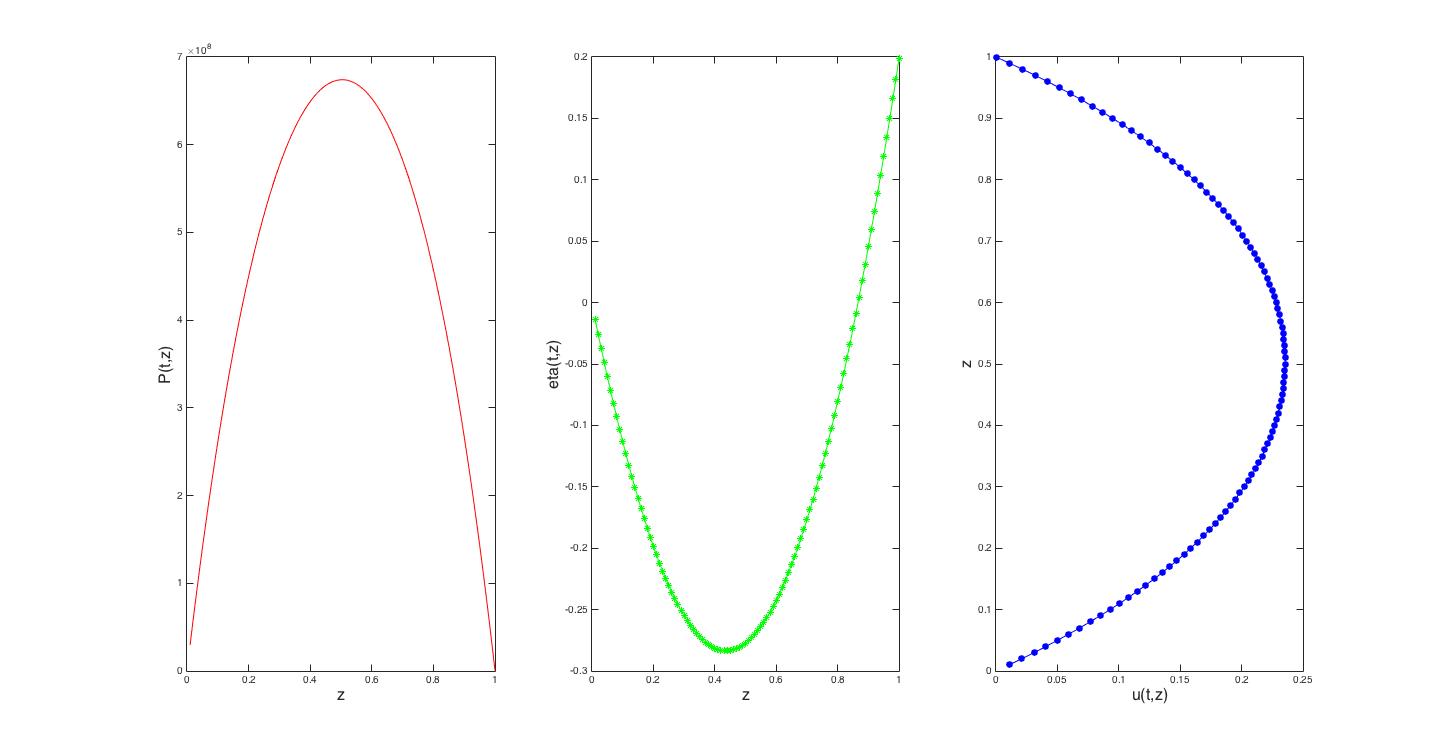}
\end{minipage}
\end{figure}
\begin{figure}[h]
\centering
\includegraphics[scale=0.125]{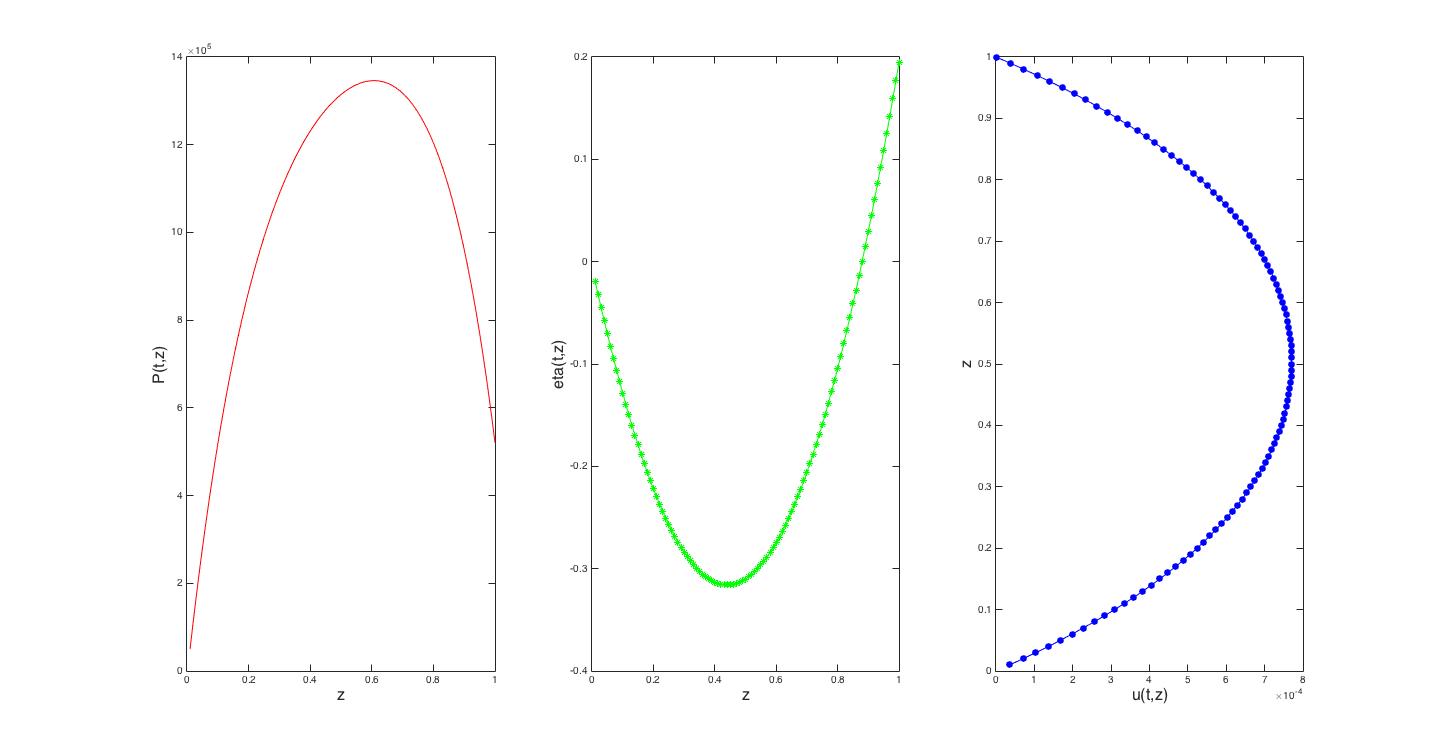}
\caption{The evolution of $P, \eta$ and $u$ with initial conditions \eqref{numer1}. }\label{fevol}
\end{figure}

In Figure \ref{fevol} we can observe in a section of length $L=1$ the profiles of the pressure (red), the tissue displacement (green) and the velocity flux (blue) at three different instants up to the final time $t=1$. 
\begin{remark}
Because of the no-slip condition, the fluid particles in the layer in contact with the surface of the compartment come to a complete stop. This layer also causes the fluid particles in the adjacent layers to slow down gradually as a result of friction. In general, to make up for this velocity reduction, the velocity of the fluid at the core of the compartment has to increase to keep the mass flow rate through the compartment constant. As a result, a velocity gradient should develop along the compartment. In our simulations we can observe, instead, a reduction in the velocity gradient at the midsection of the compartment modelized as a pipe. 
This is due to the fact that we are neglecting the compartment cross section variation which, otherwise, would concur to smooth the velocity and pressure profiles and to increase the velocity gradient without loss of energy along the compartment.
\end{remark} 

Now we want to show what happens when the global existence condition of the solution for the system \eqref{ar24} in Theorem \ref{Teorema2} are violated. In order to do that we fix the following initial data
\begin{equation}\label{numer2}
u_0(z)=-\exp(z)
\end{equation}
and for the pressure we keep the initial condition given by \eqref{ar26} and \eqref{ar26b} in Theorem \ref{Teorema} because it automatically inherits the properties of the other initial data.
\vspace{10pt}
\begin{figure}[h]
\centering
\includegraphics[scale=0.25]{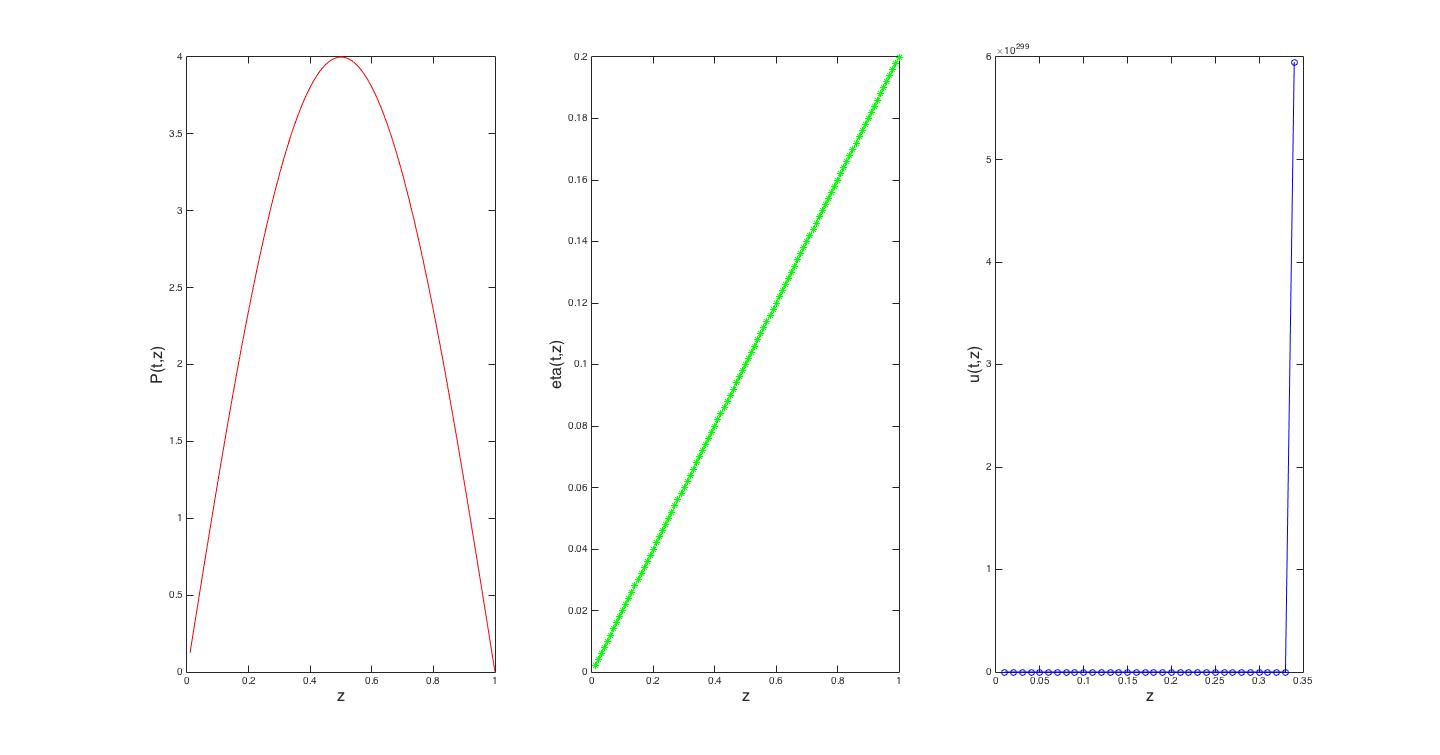}
\caption{The blow up of the velocity which occurs by violating global existence conditions. }\label{fevol1}
\end{figure}

The Figure \ref{fevol1} shows a velocity shock wave after a single iteration and consequently a blow up in pressure and displacement at the second iteration. This result is exactly what we expected according to the previous detailed analysis of the model.

Finally we can conclude that the simulations are in good agreement with the proof of a global solution for the system \eqref{ar24}, as well as with the real cerebral phenomena modelized.

\end{document}